\documentclass{nws}

\title[Triadic analysis of affiliation networks]
      {Triadic analysis of affiliation networks}

 \author[J.C. Brunson]
        {JASON CORY BRUNSON\\
         Center for Quantitative Medicine, UConn Health, Farmington, CT 06030\\
         \email{brunson@uchc.edu}}

\jdate{June 2015}
\pubyear{2015}
\pagerange{\pageref{firstpage}--\pageref{lastpage}}
\doi{}

\usepackage{color}
\def\df{\itshape}
\def\em{\slshape}

\newtheorem{example}{Example}[section]
\newtheorem{definition}[example]{Definition}
\newtheorem{lemma}[example]{Lemma}
\newtheorem{theorem}[example]{Theorem}

\newtheorem{axiom}{Axiom}

\def\id{\mathrm{id}}

\def\Hom{\mathsf{Hom}}
\def\Par{\mathsf{Par}}

\def\Z{\mathbb{Z}}
\def\Tr{{\rm Tr}}

\def\C{\mathcal{C}}

\def\T{\mathcal{T}}

\def\opsahl{^\ast}
\def\excl{^\circ}

\def\setchoose#1#2{\big\{\!{#1\atopwithdelims..#2}\!\big\}}

\makeatletter
\DeclareRobustCommand{\text}{%
  \ifmmode\expandafter\text@\else\expandafter\mbox\fi}
\let\nfss@text\text
\def\text@#1{{\mathchoice
  {\textdef@\displaystyle\f@size{#1}}%
  {\textdef@\textstyle\f@size{#1}}%
  {\textdef@\textstyle\sf@size{#1}}%
  {\textdef@\textstyle \ssf@size{#1}}%
  \check@mathfonts
  }%
}
\def\textdef@#1#2#3{\hbox{{%
                    \everymath{#1}%
                    \let\f@size#2\selectfont
                    #3}}}
\makeatother

\usepackage{tikz}
\usetikzlibrary{arrows}
\usetikzlibrary{matrix}

\usepackage{pifont}% http://ctan.org/pkg/pifont

\usepackage{url}

\usepackage{caption}

\begin{document}

\label{firstpage}

\maketitle

\begin{abstract}
Triadic closure has been conceptualized and measured in a variety of ways, most famously the clustering coefficient. Existing extensions to affiliation networks, however, are sensitive to repeat group attendance, which manifests in bipartite models as biclique proliferation. Whereas this sensitivity does not reflect common interpretations of triadic closure in social networks, this paper proposes a measure of triadic closure in affiliation networks designed to control for it. To avoid arbitrariness, the paper introduces a triadic framework for affiliation networks, within which a range of measures can be defined; it then presents a set of basic axioms that suffice to narrow this range to the one measure. An instrumental assessment compares the proposed and two existing measures for reliability, validity, redundancy, and practicality. All three measures then take part in an investigation of three empirical social networks, which illustrates their differences.
\end{abstract}

% Keywords
% affiliation networks, triadic closure, triad census, clustering coefficient, bicliques

%\tableofcontents

\section{Introduction}\label{sec:introduction}

Triadic analysis, which emphasizes the interactions within subsets of three nodes, has long been central to network science. Meanwhile, affiliation (or co-occurrence) data have long been a major source of empirical networks. Most triadic analyses of affiliation networks either collapse their higher-order structure or focus on relations among triples of nodes, often of mixed type. This paper, building upon some recent contributions, focuses instead on triples of actors, together with the non-actor structure that establishes relations among them.

\subsection{Background}\label{sec:background}

\paragraph{Precursors}\label{sec:precursors}

Previous triadic approaches in the social networks literature provide examples of hypothesis formulation, measure design, and sociological interpretation that inspired the present analysis. One thread begins with a series of studies designed to test socio-structural predictions of cognitive balance theory \cite{d-clustering1}. These predictions apply at the level of triads, but could be analyzed statistically by aggregating over an entire graph. For example, the transitive property, under which the directed relations \(p\to q\to r\) imply the relation \(p\to r\), describes social graphs with a specific hierarchical structure \cite{hl-transitivity}. While this structure would be hard to measure directly, the transitivity ratio (the global proportion of instances of \(p\to q\to r\) in which \(p\to r\)) provides a simple measure of how closely a graph respects this property \cite{hk-matrix}.

A separate thread concerns the ``small world'' property, a high concentration of ties within communities yet counterintuitively low distances between actors in different communities, observed in empirical social networks \cite{sk-contacts}. The ``strong triadic closure'' (STC) hypothesis proposed to reconcile these properties by ascribing a cohesion role to strong ties within communities and a bridging role to weak ties between them \cite{g-strength}. STC distinguishes two levels of tie (strong and weak) and posits that strong ties lead to more closures. A reorientation from triads to ego networks led to the ``structural holes'' framework, in which an actor with many weak ties, hence a more disconnected neighborhood, has increased potential as a broker. The local measure of constraint was introduced to quantify how these neighborhood connections limit brokerage potential \cite{b-structural}. A later, independent study introduced the similar but simpler clustering coefficient to quantify ``cliquishness'' across a family of small world models \cite{ws-collective}.

% BEGIN OTHER LITERATURE OMISSION
\iffalse
These studies were part of a broader program to understand the global structure of networks from local information or behavior, which itself began with triadic analysis: Triads were identified as the ``atoms'' of analysis in early socio-structural research \cite{ch-structural} and underpinned some of the earliest attempts at community detection, link prediction, and motif mining, each of which now occupies its own mature body of literature. However, triads arguably no longer occupy a central role in this program \cite{kh-triads}. An implicit goal of the present paper is, therefore, to evaluate the triadic approach itself in the AN context.
\fi
% END OTHER LITERATURE OMISSION

\paragraph{Conventions}\label{sec:conventions}

The present study concerns social networks, but the concepts generalize to any affiliation network (AN) setting. Most terminology and notation is taken from standard references \cite{bm-graph,wf-social}. Additional concepts will be defined as needed.

A graph \(G=(V,E)\) consists of a finite set \(V\) of nodes and a set \(E\subseteq V\times V\) of edges \(e=(v,w)\). Edges will be symmetric and will not include duplicates or loops. A graph is bipartite if its nodes can be partitioned into subsets \(V_1\) and \(V_2\) in such a way that \(E\subseteq V_1\times V_2\). The degree of a node \(v\) is the number of edges containing \(v\). A subgraph of \(G\) is a graph \(G'=(V',E')\) satisfying \(V'\subseteq V\) and \(E'\subseteq E\), and a subset \(W\subseteq V\) of nodes induces the subgraph \((W,E\cap(W\times W))\).

Traditional social networks consist of actors having (here, symmetric) relations among them, and are modeled as graphs with actors represented by nodes and relations by edges. Three actors, together with the relations among them, form a triad. The triads of a traditional network \(G\) take four types \(i=0,1,2,3\), according to the number of relations among their actors; the tallies \(s_i=s_i(G)\) of each type constitute the {\df triad census} \((s_0,s_1,s_2,s_3)\). The {\df (classical) clustering coefficient}, often described as the proportion of connected triples that are closed \cite{n-structure2}, is then the ratio \(C(G)={3\times s_3}/(s_2+3\times s_3)\).

Relations among the actors of an AN are established through common attendance at events; each event is attended by some subset of actors. ANs are modeled as bipartite graphs, \(V_1\) consisting of the actors and \(V_2\) the events. Though actors are only tied to events, in both settings the neighbors of an actor \(v\) shall be the actors related to \(v\). If actors who coattended events are assigned edges, then they (without the events) form a traditional social network called the projection.

\paragraph{Organization}\label{sec:organization}

Sec.~\ref{sec:exclusive} proposes the new clustering coefficient. The main body of the paper is split between theoretical (Sec.~\ref{sec:theoretical}) and empirical (Sec.~\ref{sec:empirical}) assessments of this statistic, and begin with their own organizational outlines. In short, Sec.~\ref{sec:theoretical} explores triadic analysis in the abstract, while Sec.~\ref{sec:empirical} performs triadic analyses on empirical data. Sec.~\ref{sec:conclusion} summarizes the paper, its limitations, and future directions. 
All analyses are performed, and images produced, using the open-source statistical programming language R, with the {\tt igraph} and {\tt ggplot2} packages in particular \cite{r-R,cn-igraph,w-ggplot2}. Full code is available at \url{https://github.com/corybrunson/triadic}.

\subsection{The exclusive clustering coefficient}\label{sec:exclusive}

\paragraph{Motivation}\label{sec:motivation}

``Triadic closure'' (TC) refers to the tendency for the relations \((p,q)\) and \((q,r)\) to entail the relation \((p,r)\). This entailment need not be causal, nor even chronological, but interpersonal interpretations of TC posit that the common neighbor \(q\) facilitates, or even initiates, the connection between \(p\) and \(r\). Such interpretations, however, are at odds with common measures of TC, especially in the AN setting.

The clustering coefficient, for example, is often evaluated on projections; this shall be the meaning of the shorthand \(C(G)\) when \(G\) is an AN.\footnote{\(C\) evaluates to zero on any bipartite graph.} A conspicuous feature of these projections is the proliferation of clique graphs \(K_n\), which consist of \(n\) nodes and all \({n\choose 2}\) possible edges between them. \(n\) actors in \(G\) who attend any single event produce a copy of \(K_n\) in the projection, which contains \({n\choose 3}\) 3-edge triads. These can dramatically increase \(C(G)\), so that its values are often largely determined by event size \cite{n-scientific-i,gs-analyzing}. High event attendance, however, does not guarantee TC: Individuals in distinct, pre-existing social groups at a common event may interact primarily with others in their own groups, and forge few if any inter-group relations.

Attempts to account for this inflation of \(C\) have taken both ``conversion'' (at the projection level) and ``direct'' (at the AN level) approaches. Conversion approaches have, for example, standardized the value of \(C\) by its values at a suitable null model \cite{us-collaboration} and applied clustering coefficients designed for weighted networks to weighted projections \cite{skokk-generalizations}. These methods help discriminate levels of TC among ANs, but at some cost to interpretability.

Two recent direct approaches define new clustering coefficients in terms of AN structure among triples of actors \cite{o-triadic,lr-clustering}. The {\df Opsahl clustering coefficient} \(C\opsahl\), for example, restricts the notion of ``connected triples'' (of actors) to those who are pairwise connected through distinct events. It can be defined as the proportion of 4-paths that are closed: The graph \(P_d\) consisting of distinct nodes \(v_0,v_1,\ldots,v_d\) and edges \((v_i,v_{i+1})\) is called the \(d\)-path; if, instead, \(v_0=v_d\), the result is the \(d\)-cycle \(C_d\).\footnote{The \(d\)-paths involved in this calculation must begin and end at actor nodes.} (Both have \(d\) edges; see Fig.~\ref{fig:TC}c,d.) For a 4-path in \(G\) to be ``closed'' means for it to be contained in a 6-cycle.\footnote{Several other studies have proposed bipartite clustering coefficients that concern triples of nodes but not of actors, and are not considered here \cite{o-triadic}.} In an empirical test, \(C\opsahl\) took much smaller values than \(C\), and the two statistics diverged most on the network with the greatest mean event size \cite{o-triadic}.

\begin{figure}[h]
\centerline{
\includegraphics[width=.2\textwidth, trim = 2.5cm 2.75cm 1.25cm 2.25cm, clip = false]{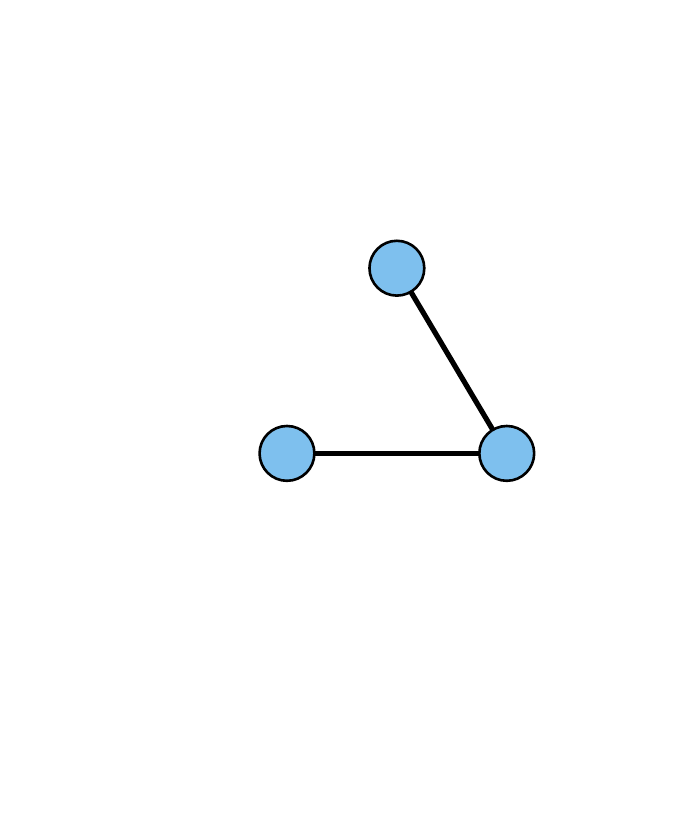}
\hspace{-2ex}a\hspace{1.5ex}\ \ 
\includegraphics[width=.2\textwidth, trim = 2.5cm 2.75cm 1.25cm 2.25cm, clip = false]{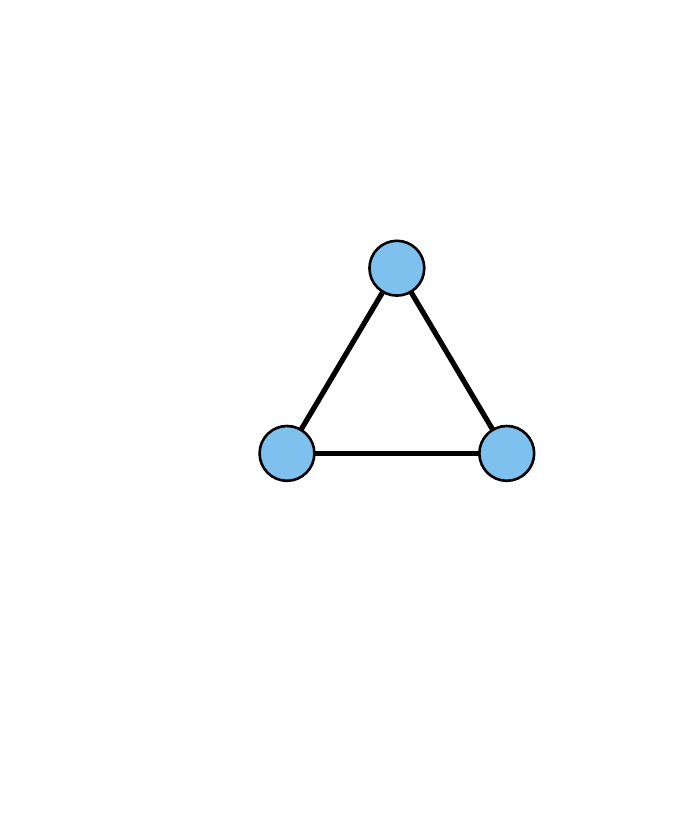}
\hspace{-2ex}b\hspace{1.5ex}\hfill
\includegraphics[width=.2\textwidth, trim = 2.5cm 2.75cm 1.25cm 2.25cm, clip = false]{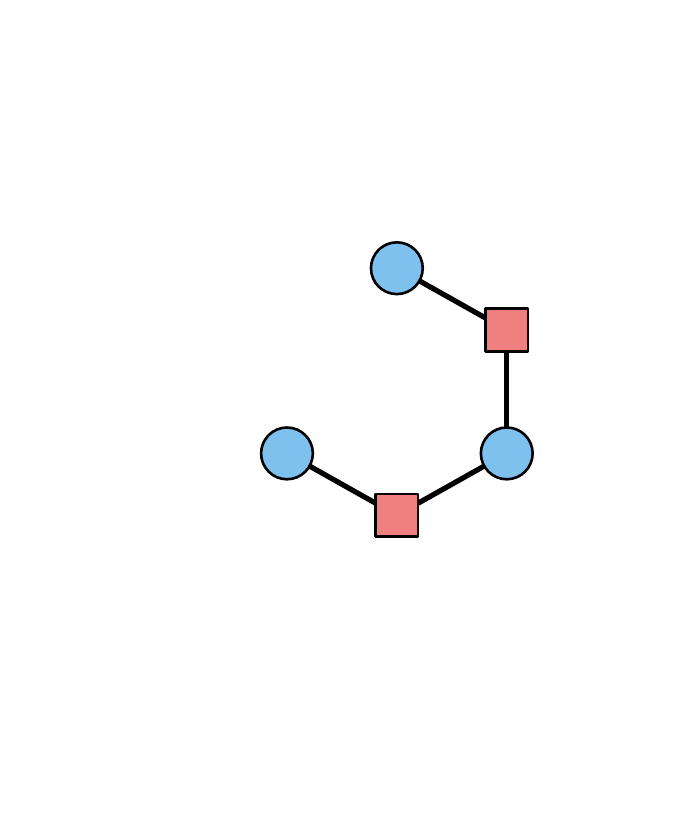}
\hspace{-2ex}c\hspace{1.5ex}\ \ 
\includegraphics[width=.2\textwidth, trim = 2.5cm 2.75cm 1.25cm 2.25cm, clip = false]{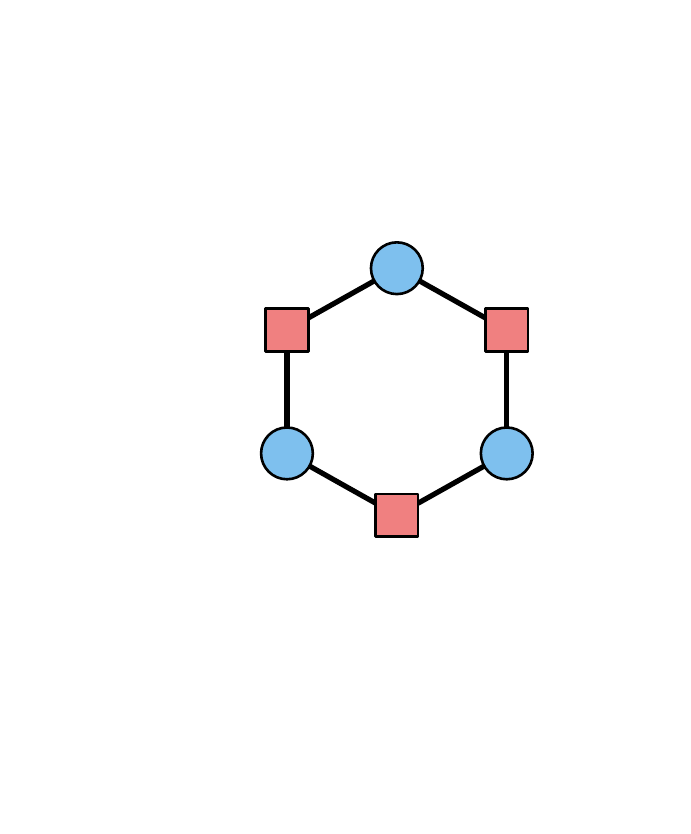}
\hspace{-2ex}d\hspace{1.5ex}
}
\caption{Traditional and affiliation network conceptions of triadic closure: In traditional networks, the 2-edge triad (a) is ``open'' while the 3-edge triad (b) is ``closed''. The clustering coefficient \(C\) is defined either as a ratio of the numbers of triads of these types (see Sec.~\ref{sec:background}) or as the proportion of subgraphs of the form (a) that are contained in a subgraph of the form (b). One extension of this idea to ANs uses the 4-path (c) and the 6-cycle (d) in place of these triads. The Opsahl clustering coefficient \(C\opsahl\) is defined as the proportion of subgraphs (c) that are contained in a subgraph (d). (Circular nodes denote actors; square nodes denote events.)}
\label{fig:TC}
\end{figure}

However, these measures may still be at odds with the popular interpretation of TC: The same pre-existing groups that attend one event are likely to attend others, though this no more entails TC than attendance at the first. Such repeat group attendance manifests in bipartite AN models as the proliferation of biclique graphs \(K_{n,m}\), which consist of \(n\) actors who each attend each of \(m\) events (hence \(n\times m\) edges). Indeed, bicliques and similar motifs have been observed in empirical ANs at frequencies greater than expected by chance \cite{be-network1,c-study}. Just as \(C\) is sensitive to cliques, \(C\opsahl\) is sensitive to bicliques: If \(m\geq 3\), then \(K_{3,m}\) contains \(6m(m-1)\) 4-paths, each of which is closed; the effect grows geometrically with \(n\).\footnote{Though note that \(K_{3,2}\) contains \(6\times 2=12\) {\em open} (and no closed) 4-paths.} Thus empirical values of \(C\opsahl\) may be dominated by patterns of repeat group attendance. The need for a measure of TC in ANs that also controls for this artifact motivated the present study.

\paragraph{Proposal}

The proposed graph statistic follows \(C\opsahl\) in restricting to pairwise connectivity through separate events within a triad. It also addresses two concerns raised by \(C\opsahl\): First, the 4-paths and 6-cycles in its calculation contain no intermediate edges---each event is attended by only two of the three actors.\footnote{This choice, and some alternatives, have received their own treatment \cite{lr-clustering}.} This eliminates the direct influence of repeat group attendance.

Another concern is how the population of actors (or of triads) should be weighted in the calculation. \(C\) weights all actors equally, in that any ordered triple of actors can have at most one 2-path through them in the projection. In contrast, many 4-paths may exist through a single ordered triple in an AN, due to a multiplicity of shared events, so that more prolific actors will tend to have more influence on the value of \(C\opsahl\). Because the present study takes an actor-centric approach, the proposed statistic is designed to weight actors equally.

The statistic is denoted \(C\excl\). It asks, {\em provided \(p\) and \(q\) attend some event without \(r\), and \(q\) and \(r\) attend some event without \(p\), what is the probability that \(p\) and \(r\) attend some event without \(q\)?} Since \(C\excl\) measures TC only through pairwise-exclusive events, it shall be called the {\df exclusive clustering coefficient}.

\begin{example}\label{ex:kite}
\begin{figure}[h]
\centerline{
\includegraphics[width=.167\textwidth, trim = 4.5cm 8.75cm 3.5cm 8cm, clip = true]{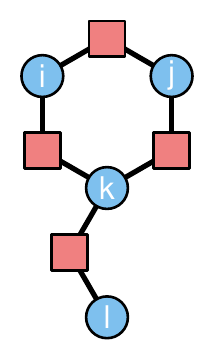}
\hspace{-3ex}a\hspace{1.5ex}\hfill
\includegraphics[width=.167\textwidth, trim = 4.5cm 8.75cm 3.5cm 8cm, clip = true]{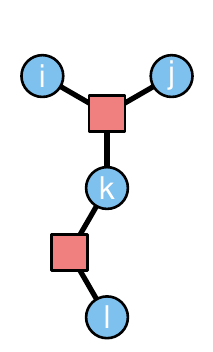}
\hspace{-3ex}b\hspace{1.5ex}\hfill
\includegraphics[width=.167\textwidth, trim = 4.5cm 8.75cm 3.5cm 8cm, clip = true]{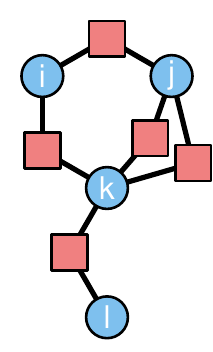}
\hspace{-3ex}c\hspace{1.5ex}\hfill
\includegraphics[width=.167\textwidth, trim = 4.5cm 8.75cm 3.5cm 8cm, clip = true]{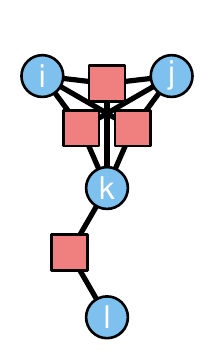}
\hspace{-3ex}d\hspace{1.5ex}\hfill
\includegraphics[width=.167\textwidth, trim = 4.5cm 8.75cm 3.5cm 8cm, clip = true]{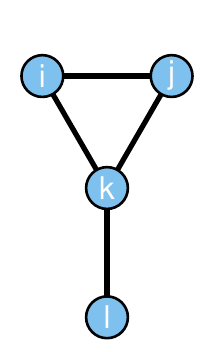}
\hspace{-3ex}e\hspace{1.5ex}
}
\caption{Four ANs (a--d) having the same projection (e).}
\label{fig:kite}
\end{figure}

Fig.~\ref{fig:kite} depicts four ANs that project to the same the ``kite'' graph. AN (a) exhibits TC in the sense of interest, while (b) exhibits TC of the kind \(C\opsahl\) was designed to ignore. \(C\opsahl\) evaluates to \(\frac{3}{5}\) at (a) and to \(0\) at (b), and \(C\excl\) agrees on both. \(C\opsahl\) evaluates to \(\frac{5}{8}\) at (c), and to \(\frac{3}{4}\) at (d), due to additional copies of \(P_4\) and \(C_6\). For instance, six copies of \(P_4\) in (d) proceed from \(i\) through \(j\) to \(k\), and each is closed. In contrast, \(C\excl\) takes the familiar values \(\frac{3}{5}\) at (c) and \(0\) at (d), since it is calculated on the same numbers of distinct 4-paths and 6-cycles.
\end{example}

%\paragraph{Interpretation}

By eliminating sources of 3-edge triads other than the popular meaning of TC, \(C\excl\) may help to infer dynamic information from static data. The popular meaning is dynamic: Actors who are not neighbors, but who have neighbors in common at one time, become neighbors at a later time \cite{ek-networks,mbkn-coauthorship}. If a traditional network \(G\) has edges labeled by instants in time, such that an edge labeled \(t\) is said to exist after \(t\) but not before, define the {\df dynamic triadic closure} \(D(G)\) to be, among those triads at which there is at some time an open 2-path, the proportion at which there is at a later time a 3-cycle. If \(G\) is an AN with events labeled by time, then \(D\) shall be calculated on its projection, where each edge is labeled by the earliest event that projects to it.

In the traditional setting, if a network has time-labeled edges, no two of which are simultaneous, then \(D=s_3/(s_2+s_3)=C/(3-2\times C)\). In the AN setting, pairwise-exclusive events are essential to \(D\), since an open 2-path in the projection must correspond to a triad with only pairwise-exclusive events. While the two calculations are in general unequal even when no two events are simultaneous, \(C\excl\) could provide a useful estimate of \(D\).
%The problem of inferring the incidence of TC from static AN data may itself be a necessary first step toward inferring its chronology.

\section{Theoretical analyses}\label{sec:theoretical}

This section formalizes the exclusive clustering coefficient and evaluates its theoretical merits. Sec.~\ref{sec:classification} develops a formal notion of ``triad'' for ANs. On this foundation, Sec.~\ref{sec:category} unifies \(C\), \(C\opsahl\), and \(C\excl\) into a generic clustering coefficient. This definition specializes to impracticably many statistics, which Sec.~\ref{sec:axiomatic} whittles down by appeal to several properties suited to present purposes. The technical details of this process are relegated to Sec.~\ref{sec:proofs}.

\subsection{Triads}\label{sec:classification}

\paragraph{Scheduled subgraphs}

A triad-centric approach to ANs requires an object of study. What, then, is a ``triad''? Since the triads of a traditional network are the subgraphs {\em induced} by three actors, a suitable analog of induced subgraphs for ANs would suffice. This paper proposes to include those events that establish relations among a set of actors:

\begin{definition}\label{def:scheduled}
Given an AN \(G\) and a subset \(W\) of actors of \(G\), the subgraph of \(G\) {\df scheduled} by \(W\) is the subgraph induced by the actors \(W\) together with all events attended by at least two actors in \(W\).
{\df Scheduled} graph maps are defined analogously to induced graph maps, and the {\df triads} of \(G\) are the subgraphs scheduled by three actors.
\end{definition}

\begin{figure}[h]
\centerline{
\includegraphics[width=.3\textwidth, trim = 3cm 8cm 2cm 7.5cm, clip = true]{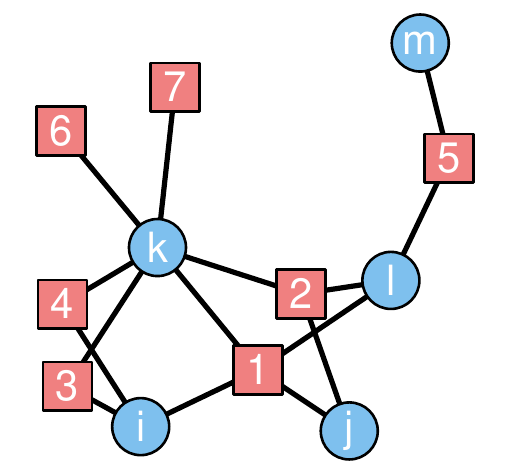}
\hspace{-3ex}a\hspace{1.5ex}\hfill
\includegraphics[width=.3\textwidth, trim = 3cm 8cm 2cm 7.5cm, clip = true]{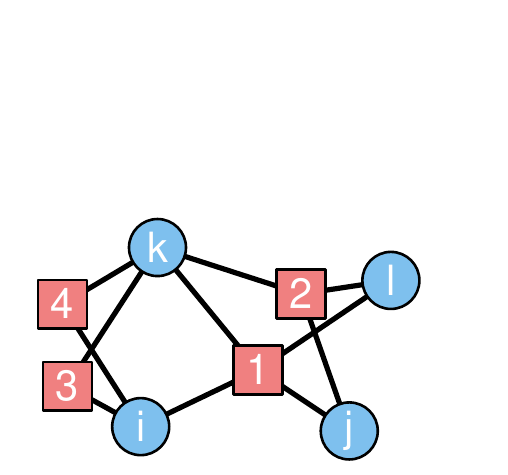}
\hspace{-3ex}b\hspace{1.5ex}\hfill
\includegraphics[width=.3\textwidth, trim = 3cm 8cm 2cm 7.5cm, clip = true]{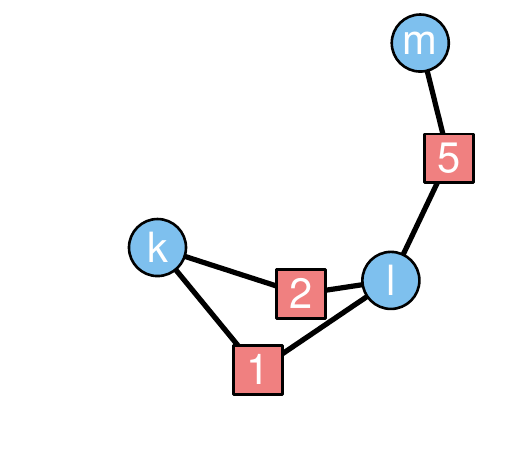}
\hspace{-3ex}c\hspace{1.5ex}
}
\caption{An AN (a) of actors \(i,j,k,l,m\) and its scheduled subgraphs (b) at \(\{i,j,k,l\}\) and (c) at \(\{k,l,m\}\).}
\label{fig:sched}
\end{figure}

\begin{example}\label{ex:sched}
Fig.~\ref{fig:sched} depicts an AN of five actors and two of its scheduled subgraphs. Note in particular that the scheduled subgraph on the entire set of actors (not shown) does not include events 6 and 7, since they play no role in establishing relations among the actors.\footnote{This shows that an AN need not be the scheduled subgraph of its actors, contrary to the analogous property of induced subgraphs. Their projections, however, are the same (up to edge weights).}
\end{example}

\paragraph{Triad censuses}

The classification of AN triads is straightforward but not trivial. While traditional triads fall into four isomorphism classes (see Sec.~\ref{sec:category}), AN triads, in theory, occupy arbitrarily many, due to the unlimited number of events two or three actors might attend. Consider an arbitrary triad with actors \(p,q,r\). Take \(w_{pq}\) to be the number of events attended by \(p\) and \(q\), similarly define \(w_{qr}\) and \(w_{pr}\), and take \(w_{pqr}\) to be the number of events attended by all three. (Note that \(w_{pq}\) does not depend on \(r\), etc.) Up to isomorphism, it may be assumed that \(w_{pq}\geq w_{qr}\geq w_{pr}\) (otherwise relabel the actors). Necessarily, \(w_{pqr}\leq w_{pq}\). Let \(\mu=(\mu_1,\mu_2,\mu_3)=(w_{pq}-w_{pqr},w_{qr}-w_{pqr},w_{pr}-w_{pqr})\) count the ``exclusive'' events between each pair of actors, and let \(w=w_{pqr}\) count the ``inclusive'' events attended by all three. The pair \((\mu,w)\) determines the isomorphism class of the triad.\footnote{While this scheme is more intuitive, a more storage-friendly enumeration of the triad classes is given by the quadruple \((w_{pqr}-w_{pq},w_{pq}-w_{qr},w_{qr}-w_{pr},w_{pr})\in(\Z_{\geq 0})^4\).} Since \(\mu_1\geq\mu_2\geq\mu_3\), \(\mu\) is an integer partition of three parts; write \(\mu\in\Par_3\). Where \(\Z_{\geq 0}\) is the set of nonnegative integers and \(\T\) is the set of triad isomorphism classes, this gives a bijective correspondence
\[\T\leftrightarrow\Par_3\times\Z_{\geq 0}\text.\]
Write \(\Tr_{\mu w}\) for the triad described above, and \(s_{\mu w}=s_{\mu w}(G)\) for the number of triads of \(G\) isomorphic to \(\Tr_{\mu w}\). The {\df (full) triad census} of \(G\) is then the array \((s_{\mu w})_{\mu,w}\). The partitions \(\Par_3\) can be totally ordered, and thereby the census arranged in a matrix, whose size depends on the network.\footnote{Where \(n=\max(\mu)\), there are bijections \(\sigma:\Par_3^{(n)}\to\setchoose{n+3}{3}\), from the partitions in \(\Par_3\) having parts \(\leq n\) to the subsets of \(\{1,\ldots,n+3\}\) of size 3, and \(\rho:\setchoose{n+3}{3}\to\{1,\ldots,{{n+3}\choose{3}}\}\), which indexes these subsets; the composition \(\rho\circ\sigma:\Par_3^{(n)}\to\{1,\ldots,{{n+3}\choose{3}}\}\) indexes the partitions.
\(\sigma\) is a classical bijection \cite{s-enumerative};
%First identify a partition \(\lambda\) with its {\df Ferrers diagram}, in the English convention, inside the \(n\times 3\) rectangle. Then record the sequence of south and west moves that trace its border from the corner of the rectangle at \((3,0)\) to the opposite corner at \((0,-n)\). The positions of the (three) west moves give the subset of \(\{1,\ldots,n+3\}\) of size 3 \cite{s-enumerative}.
\(\rho\) is the {\df revolving door ordering} \cite{ks-combinatorial}.} Necessarily, \(\sum_{\mu,w}s_{\mu w}={{|V_1|}\choose 3}\). The triads scheduled from \(i,j,k\) in Fig.~\ref{fig:kite} (a--d), for example, are \(\Tr_{(1,1,1),0}\), \(\Tr_{(0,0,0),1}\), \(\Tr_{(2,1,1),0}\), and \(\Tr_{(0,0,0),3}\).

This scheme explodes as networks grow dense. The following alternative scheme is instead bounded, but nonetheless captures useful affiliation structure: The events of a triad fall into four structural equivalence classes, according to which actors attend them. Instead of binning triads by {\em how many} events they have in each class, bin them by whether they contain {\em some} event in each class. If \(\Tr_{\mu w}\) has, for example, any inclusive event (i.e., if \(w>0\)), then \(\Tr_{\mu w}\) shares a bin with \(\Tr_{\mu,1}\); otherwise it is \(\Tr_{\mu,0}\). Each bin then contains exactly one representative \(\Tr_{\mu w}\) with \(\mu_1,\mu_2,\mu_3,w\in\{0,1\}\), and this bin is determined by the two numbers \(x=\mu_1+\mu_2+\mu_3\in\{0,1,2,3\}\) and \(y=w\in\{0,1\}\). The {\df structural triad census} consists of the eight tallies \(t_{xy}\) of triads in each bin. Though containing only twice as many bins as the simple census, the structural census contains useful additional information (see Thm.~\ref{thm:binning} and Sec.~\ref{sec:instrument}).

\begin{figure}[h]
  \centerline{
    \includegraphics[width=.32\textwidth, trim = 2.75cm 2cm 1.75cm 2cm, clip = true]{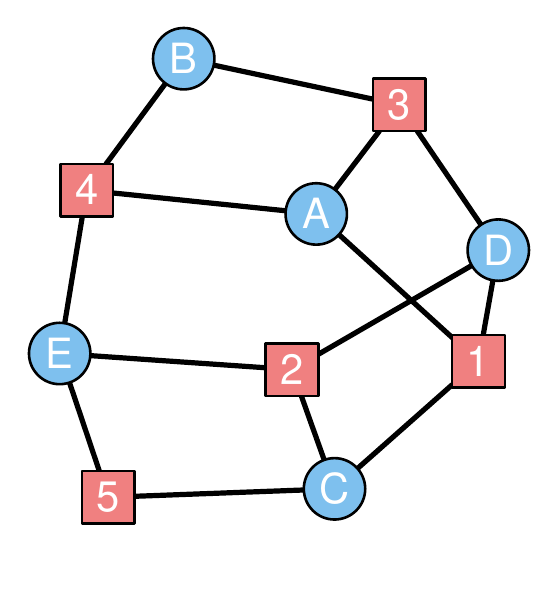}
    \hspace{-3ex}a\hspace{8ex}
    \begin{minipage}[t]{.2\textwidth}
      \vspace{-32ex}
      \mbox{\footnotesize\(\displaystyle\begin{array}{l|rr}
      \input{tab-dg2-census.txt}
      \end{array}\)}
    \end{minipage}
    \hspace{0ex}b\hspace{8ex}
    \includegraphics[width=.32\textwidth, trim = 2.75cm 2cm 1.75cm 2cm, clip = true]{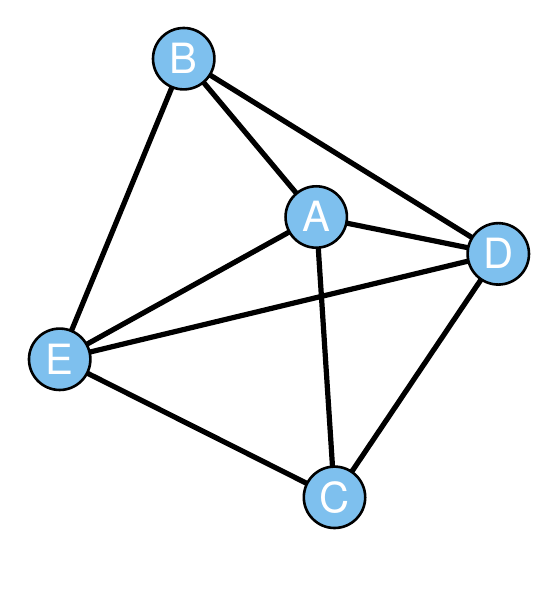}
    \hspace{-3ex}c\hspace{2.5ex}
  }
  \caption{The network DG2 (a), its full triad census (b), and its projection (c). The column in (b) indicates the number of inclusive events; the row indicates the distribution of exclusive events across pairs of actors. For example, the triad at \((\text{A},\text{B},\text{C})\) is tallied in column \(0\), row \((2,1,0)\) (see Ex.~\ref{ex:dg2}).}
  \label{fig:dg2}
\end{figure}

\iffalse
\begin{table}
  \caption{Full triad census for DG2.}
  \label{tab:dg2census}
  \begin{minipage}{\textwidth}
  \input{tab-dg2-census.txt}
  \end{minipage}
\end{table}
\fi

\begin{example}\label{ex:dg2}
The network DG2, depicted in Fig.~\ref{fig:dg2} with its full census and its projection, is taken from a famous study of the American racial caste system \cite{dgg-deep}. As an example of a social unit, the study presented attendance data for five acquainted women (``Miss A'' through ``Miss E'') and five social activities (bridge, dinner, movies, dance, and visiting), forming an AN. The projection contains three 2-edge and seven 3-edge triads, so the simple census is \((0,0,3,7)\). (Therefore, incidentally, \(C(\text{DG2})=\frac{3\times 7}{3+3\times 7}=\frac{7}{8}\).)

These tallies obscure higher-order structure: The seven fully-connected triads fall into four classes. One might be called ``symmetric'' and ``exclusive'': Mss.~B, D, and E attended no events together, but each pair were present at one, so that their triad is (isomorphic to) \(\Tr_{(1,1,1),0}\). Two triads were exclusive but not symmetric: Mss.~C and E attended two events without Ms.~A, though Ms.~A attended different, separate events with Mss.~C and E; they thus form a triad \(\Tr_{(2,1,1),0}\), as do Mss.~A, D, and E. The remaining four were ``inclusive'', in that all three women attended some event together (specifically, the four activities of attendance 3). In each case, at least one pair of women attended another event together, forming the triads \(\Tr_{(1,0,0),1}\) and \(\Tr_{(1,1,0),1}\). The women constituting each of the three 2-edge triads joined in no single activity together, instead forming three copies of \(\Tr_{(2,1,0),0}\). For example, Mss.~A and B attended two events together, Mss.~A and C one, and Mss.~B and C none. (As an exercise, the reader might recover the structural census from Fig.~\ref{fig:dg2}a,b.)
\end{example}

\subsection{Category framework}\label{sec:category}

\paragraph{Graph maps}

A generic clustering coefficient will be defined in terms of graph maps. For present purposes, a graph map \(\phi:G\to H\) (``\(\phi\) from \(G\) to \(H\)'') shall assign each node \(v\) of \(G\) to a node \(\phi(v)\) in \(H\) (the image of \(v\) under \(\phi\)) in such a way that every edge \((v,w)\) in \(G\) is preserved, i.e.\ \((\phi(v),\phi(w))\) is an edge in \(H\). One example is the inclusion of a subgraph \(G\subseteq H\). A graph map \(\phi:G\to H\) is called induced if the image \(\phi(G)\subseteq H\) is an induced subgraph. The images \(\phi(v)\) in \(H\) and the preserved edges among them form the image of \(G\) in \(H\). Two graph maps \(\phi:G\to H\) and \(\psi:H\to K\) yield the composition \(\psi\circ\phi:G\to K\) defined by \(\psi\circ\phi(v)=\psi(\phi(v))\). Such a graph map \(\psi\circ\phi:G\to K\) is said to factor through \(H\); for example, any map \(\phi:G\to H\) factors through its image \(\phi(G)\subseteq H\).

A graph map \(\phi:G\to H\) is injective if it sends no two nodes in \(G\) to the same node in \(H\), and surjective if every node in \(H\) is the image of some node in \(G\) (its pre-image); by a ``copy'' of \(G\) in \(H\), or a path or cycle ``in \(G\)'', shall be meant the image of an injective map. (By convention, paths and cycles in an AN arise from maps that send \(v_0\) to an actor.) Thus a 4-path \(\phi:P_4\to G\) is closed if it factors through \(C_6\).

An injective, surjective map is bijective, and a bijective map \(\phi:G\to H\) is an isomorphism if it is induced---that is, if it preserves absences of edges (\((\phi(v),\phi(w))\notin G\) whenever \((v,w)\notin H\)). The isomorphisms establish an equivalence relation on graphs; two graphs related by an isomorphism are said to be isomorphic, and to lie in the same isomorphism class. Two nodes \(v,w\in G\) are structurally equivalent if there is an isomorphism \(G\to G\) that sends every node to itself except exchanges \(v\) and \(w\); this establishes an equivalence relation on the nodes of \(G\).

\paragraph{Categories}

The framework of category theory, though not necessary, absorbs some useful and unobjectionable yet messy assumptions into the notation, provides a catalogue of natural examples, and avoids unnecessary constraints on the range of possibilities.\footnote{While there are infinitely many AN triads, their combinatorial complexity is limited (see Sec.~\ref{sec:classification}). It would be short work to classify a useful collection (19, by the author's count) of clustering coefficients, in the sense of Def.~\ref{def:Chat} and including \(C\), \(C\opsahl\), and \(C\excl\), by which structural equivalence classes of events the events of \(W\) and \(X\) may be mapped to, and which of these should then be considered congruent. This scheme, however, would omit more ad hoc clustering coefficients, for instance one that requires the events \(v_1,v_3\) of \(W\) to be mapped to exclusive events but places no such constraint on \(v_5\) in \(X\). Such a statistic would violate Axiom~\ref{ax:induced}, but may be very useful in certain settings (compare to the discussion of STC in Sec.~\ref{sec:properties}).}

A {category} \(\C\) consists of a set of {objects}; for each pair of objects \(A,B\), a set \(\Hom_{\C}(A,B)\) of {morphisms} from \(A\) to \(B\); and, for each pair of morphisms \(f\in\Hom_\C(B,C)\) and \(g\in\Hom_\C(A,B)\), the {composition} \(f\circ g\in\Hom_\C(A,C)\); all subject to the following conditions \cite{m-theory}:
\begin{enumerate}
  \renewcommand{\theenumi}{\roman{enumi}}
  \item\label{item:identity} (Identity) For every \(A\in\C\) there exists \(\id_A\in\Hom_\C(A,A)\) satisfying \(f\circ \id_A=f\) and \(\id_A\circ g=g\) for any \(f\in\Hom_\C(A,B)\) or \(g\in\Hom_\C(C,A)\).
  \item\label{item:associativity} (Associativity) For any triple of morphisms \(f\in\Hom_\C(C,D)\), \(g\in\Hom_\C(B,C)\), and \(h\in\Hom_\C(A,B)\), \(f\circ(g\circ h)=(f\circ g)\circ h\).
\end{enumerate}
A {subcategory} \(\C'\subseteq\C\) consists of the same objects as \(\C\) and subsets \(\Hom_{\C'}(A,B)\subseteq\Hom_\C(A,B)\) that also form a category.
A {congruence relation} \(\sim\) on \(\C\) consists of equivalence relations \(\sim_{A,B}\) on each \(\Hom_\C(A,B)\) that are compatible with the composition of morphisms, so that the {quotient category} \(\C/\sim\) is determined by the objects of \(\C\) and the equivalence classes of morphisms of \(\C\) under \(\sim\).
%Categories derived from \(\C\) by taking subcategories and quotient categories shall be called {descendants} of \(\C\).
%(Note that the objects of any descendant of \(\C\) are the objects of \(\C\).)

Henceforth, view \(\T\) as the category of AN triads, with morphisms the graph maps \(\phi:H\to K\) that assign the actors of \(H\) to distinct actors of \(K\) (and therefore send events only to events), and with composition given by \((f\circ g)(v)=f(g(v))\).
%Take \(\G\) to be the category of affiliation networks with morphisms those graph maps that send actors to actors and are injective on actors, with composition given by \((f\circ g)(v)=f(g(v))\).
%Take \(\g\) to be the category of traditional networks with morphisms the injective graph maps.
\(\T\) can be viewed as a subcategory of the category of graphs \cite{h-introduction} (with many objects omitted).
Write \(\Hom_\T^K(G,H)\) for the set of morphisms from \(G\) to \(H\) that factor through \(K\). If \(G\) is any AN, write \(\Hom_\T(H,G)\) (an abuse of notation) for the set of all morphisms from \(H\) to any triad of \(G\).

\paragraph{Clustering coefficients}\label{sec:coefficients}

All three clustering coefficients described in Sec.~\ref{sec:exclusive} are expressible in category-theoretic terms. %First, note that the classical clustering coefficient of a traditional network \(G\) can be realized using either of the formulations
%\begin{equation}\label{eq:C}
%C(G)
%\ \ =\ \ \frac{|\Hom_\g^{C_3}(P_2,G)|}{|\Hom_\g(P_2,G)|}
%\ \ =\ \ \frac{|\Hom_\g(C_3,G)|}{|\Hom_\g(P_2,G)|}\text,
%\end{equation}
%where \(\Hom^K(G,H)\) denotes the set of morphisms from \(G\) to \(H\) that factor through \(K\).
%The evaluation of \(C\) on the projection of an AN can be similarly formulated.
Let \(\approx\) denote the congruence relation on \(\T\) given by taking any two maps that agree on actors to be congruent. For example, there is only one graph map from \(P_4\) to the kite graph (a) in Fig.~\ref{fig:kite} that sends \(v_0,v_2,v_4\) to \(i,j,k\) (respectively), and likewise only one such map to (b). However, there are several such maps to (c), which are all congruent in \(\T/\approx\). Thus \(\approx\) is a ``strong'' relation in that it relates very many morphisms. It turns out that, for an affiliation network \(G\),
\begin{equation}\label{eq:Caff}
C(G)
\ \ =\ \ \frac{|\Hom_{\T/\approx}^{C_6}(P_4,G)|}{|\Hom_{\T/\approx}(P_4,G)|}
\ \ =\ \ \frac{|\Hom_{\T/\approx}(C_6,G)|}{|\Hom_{\T/\approx}(P_4,G)|}\text.
\end{equation}

The Opsahl clustering coefficient restricts the morphisms in Eq.~\ref{eq:Caff} to injective graph maps. It is straightforward to check that these form a subcategory \(\widetilde\T\subset\T\). No congruence relation was imposed; for consistency of notation, write \(\widetilde\T/=\) for \(\widetilde\T\), where \(=\) denotes equality of graph maps (the weakest possible relation). \(C\opsahl\) is then realized as
\begin{equation}\label{eq:Copsahl}
C\opsahl(G)
\ \ =\ \ \frac{|\Hom_{\widetilde\T/=}^{C_6}(P_4,G)|}{|\Hom_{\widetilde\T/=}(P_4,G)|}\text,
\end{equation}
analogously to the first formulation in Eq.~\ref{eq:Caff}.
The present proposal further restricts the morphisms in Eq.~\ref{eq:Copsahl} to induced injections. These turn out to form their own subcategory \(\overline\T\subset\widetilde\T\).
Additionally, the graph maps that agree on actors {\em and that send events to structurally equivalent images} constitute a congruence relation \(\simeq\) on \(\overline\T\) (or \(\T\)), which is weaker than \(\approx\) but stronger than \(=\). The statistic \(C\excl\) is then realized as
\begin{eqnarray}\label{eq:Cexcl}
C\excl(G)
& =& \frac{|\Hom_{\overline\T/\simeq}^{C_6}(P_4,G)|}{|\Hom_{\overline\T/\simeq}(P_4,G)|}
\ \ =\ \ \frac{|\Hom_{\overline\T/\simeq}(C_6,G)|}{|\Hom_{\overline\T/\simeq}(P_4,G)|}\text.
\end{eqnarray}
%in the quotient category \(\overline\T/\simeq\).

\subsection{Axiomatic approach}\label{sec:axiomatic}

\paragraph{General formulation}

What is a ``clustering coefficient'', especially in the AN setting? Sec.~\ref{sec:category} formulated three variations on the idea, and this section presents a single unifying definition.

The statistics \(C\) and \(C\opsahl\) differ in three respects: the choice between the formulations in Eq.~\ref{eq:Caff} (which sometimes agree), the subcategory of graph maps from which the morphisms in Eq.~\ref{eq:Caff} are drawn, and the congruence relation imposed on them. Whereas \(P_4\) (isomorphic to \(\Tr_{(1,1,0),0}\)) and \(C_6\) (isomorphic to \(\Tr_{(1,1,1),0}\)) are now recognizable as two among an infinite collection of triads (see Fig.~\ref{fig:triad}), a fourth choice presents: What makes a triple of actors ``open'' or ``closed''? Another direct approach \cite{lr-clustering} considered three alternatives to \(C_6\): \(\Tr_{(1,1,0),1}\), \(\Tr_{(1,0,0),2}\), and \(\Tr_{(0,0,0),3}\). (These are the four AN triads whose duals are also triads, and in fact are self-dual \cite{b-duality}.) Alternatives to \(P_4\), sometimes taken in pairs, were obtained by removing a single event from these. The four choices thus outlined are incorporated into the following general definition:

\begin{definition}\label{def:Chat}
Pick canonical triads \(X\in\T\) and \(W\subset X\), a canonical subgraph relation \(\iota:W\to X\) (there may be many), a subcategory \(\C\subseteq\T\), and a congruence relation \(\sim\) on \(\C\). A {\df (global) clustering coefficient} of \(G\) shall be a statistic of either form
\begin{eqnarray}\label{eq:crate}
\widehat{C}(G)
& = & \displaystyle\frac{|\Hom_{\C/\sim}^X(W,G)|}{|\Hom_{\C/\sim}(W,G)|}
\ \ \ \ \text{(``rate of wedge closure'') or} \\\label{eq:mratio}
\widehat{C}(G)
& = & \displaystyle\frac{|\Hom_{\C/\sim}(X,G)|}{|\Hom_{\C/\sim}(W,G)|}
\ \ \ \ \text{(``alcove-to-wedge ratio'')\text,}
\end{eqnarray}
where morphisms factor through \(X\) only via \(\iota\). Call the morphisms \(\Hom_{\C/\sim}(W,G)\) the {\df wedges} of \(G\)---{\df closed} if they factor through \(X\), {\df open} if not---and \(\Hom_{\C/\sim}(X,G)\) the {\df alcoves} of \(G\).

Further designate a {\df center} actor \(v_c\in\{p,q,r\}\) in (each) \(W\). Given an actor \(j\in G\), obtain the {\df (local) clustering coefficient} \(\widehat{C}(j)\) of \(j\) by requiring of the morphisms in Eq.~\ref{eq:crate} or \ref{eq:mratio} that \(\phi(v_c)=j\) and \(\psi(\iota(v_c))=j\)---that is, that wedges and alcoves are {\df centered} at \(j\). The {\df wedge-dependent clustering coefficient} \(\widehat{C}_\ell\) of an affiliation network \(G\) shall be the mean value of \(\widehat{C}(j)\) across the actors \(j\) at which exactly \(\ell\) wedges are centered.
\end{definition}

By letting \(X\) range over the four self-dual triads; \(\C\) over \(\T\supseteq\widetilde\T\supseteq\overline\T\); \(\sim\) over \(=\), \(\simeq\), and \(\approx\); and adopting either Eq.~\ref{eq:crate} or \ref{eq:mratio}, Def.~\ref{def:Chat} specializes to \(4\times 3\times 3\times 2=72\) fairly straightforward statistics, which include \(C\), \(C\opsahl\), and \(C\excl\).\footnote{Some of these turn out to be the same statistic; for example, assuming \(X=\Tr_{(1,1,1),0}\), the category choices \(\overline\T/\simeq\) and \(\overline\T/\approx\) both yield \(C\excl\).} For present purposes, the best choice of \(X\) is clearly \(\Tr_{(1,1,1),0}\), leaving \(W=\Tr_{(1,1,0),0}\). These choices are assumed henceforth. (Note, however, that Thm.~\ref{thm:census} does not require this assumption.)

\begin{table}[h]
  \caption{Three measures of global and local triadic closure in DG2.}
  \label{tab:dg2local}
  \begin{minipage}{\textwidth}
  % latex table generated in R 3.1.0 by xtable 1.7-4 package
% Sun May 24 18:54:19 2015
\begin{tabular}{crrrrrr}
  \hline
\hline
 & DG2 & Miss A & Miss B & Miss C & Miss D & Miss E \\ 
  \hline
Classical & \(0.875\) & \(0.833\) & \(1.000\) & \(1.000\) & \(0.833\) & \(0.833\) \\ 
  Opsahl & \(0.611\) & \(0.500\) & \(0.667\) & \(0.667\) & \(0.600\) & \(0.714\) \\ 
  Exclusive & \(0.600\) & \(0.500\) & \(1.000\) & \(0.500\) & \(0.500\) & \(0.750\) \\ 
   \hline
\hline
\end{tabular}

  \end{minipage}
\end{table}

\begin{example}
Evaluations of \(C\), \(C\opsahl\), and \(C\excl\) at DG2 (Table~\ref{tab:dg2local}) are illustrative: Each pair of women differ by at least one statistic, implying that they all occupy structurally distinct neighborhoods; none of the statistics, however, distinguishes them all. While \(C\opsahl\) and \(C\excl\) take lower values than \(C\), the rankings of the actors are loosely correlated. Of particular interest are Mss.~B and C, whom \(C\opsahl\) and \(C\) do not distinguish but who take opposite values of \(C\excl\). At Miss B, the 4-path \((A,3,B,4,E)\) is an open wedge to \(C\opsahl\) but not a wedge at all to \(C\excl\); at Miss C, the 4-path \((D,1,C,5,E)\) is as a wedge to both \(C\opsahl\) and \(C\excl\) but only closed to \(C\opsahl\).

\(C\opsahl\) attributes high TC to Miss C because her friends remain better-connected when she is removed from the network, {\em while the events she attended remain}. In contrast, \(C\excl\) attributes high TC to Miss B because her friends remain better-connected when she is removed from the network {\em along with the events she attended}. The statistic \(C\opsahl\) thus detects TC that relies in part on inclusive events, which \(C\excl\) does not.
\end{example}

The remainder of this section comes with a warning that the labeling schemes for triad nodes vary by context: Canonical triads \(\Tr_{\mu w}\) have actors \(p,q,r\) such that \(w_{pq}\geq w_{qr}\geq w_{pr}\) (and unlabeled events); \(W\) and \(X\) adopt the schemes \(v_0,v_1,\ldots\) for \(P_4\) and \(C_6\) from Sec.~\ref{sec:exclusive}; and triads in larger ANs are scheduled at (ordered) triples of actors \((i,j,k)\) with events \(a,b,\ldots\).

\paragraph{Axioms}

Sec.~\ref{sec:exclusive} delineated three goals for a new clustering coefficient: account for event size, as \(C\opsahl\) does but \(C\) does not; further account for repeat group attendance, as neither \(C\) nor \(C\opsahl\) do; and weight actors equally, as \(C\) does but \(C\opsahl\) does not. This section wraps these desiderata into four axioms on \(\widehat C\). These are not suited to all settings, but they do help organize the myriad statistics that arise from Def.~\ref{def:Chat}.

% BEGIN WEDGE-ALCOVE AXIOM OMISSION
\iffalse
The following two axioms articulate natural expectations of the canonical alcove (and wedge(s)). Axiom~\ref{ax:6cycle} makes it possible for \(\widehat C\) to detect the kind of TC, involving exclusive events, emphasized in Sec.~\ref{sec:exclusive}, while Axiom~\ref{ax:duality} requires that \(X\) consist only of three events, and that every actor attends at least two. Together, the axioms constrain the choices to those of Opsahl: \(X=C_6\) and \(W=P_4\).

\begin{axiom}[Exclusive closure]\label{ax:6cycle}
There exists a graph map from \(X\) to \(C_6\).
\end{axiom}

\begin{axiom}[Duality]\label{ax:duality}
\(X\) remains a triad after reversing the roles of actors and events \cite{b-duality}.
\end{axiom}
\fi
% END WEDGE-ALCOVE AXIOM OMISSION

% BEGIN RESTRICTION AXIOM OMISSION
\iffalse

\begin{axiom}[Schedule restrictions]\label{ax:restriction}
Morphisms factor through the schedules of their images.
\end{axiom}
\fi
% END RESTRICTION AXIOM OMISSION

\begin{figure}[h]
\centerline{
\includegraphics[width=.2\textwidth, trim = 2.5cm 2.75cm 1.25cm 2.25cm, clip = false]{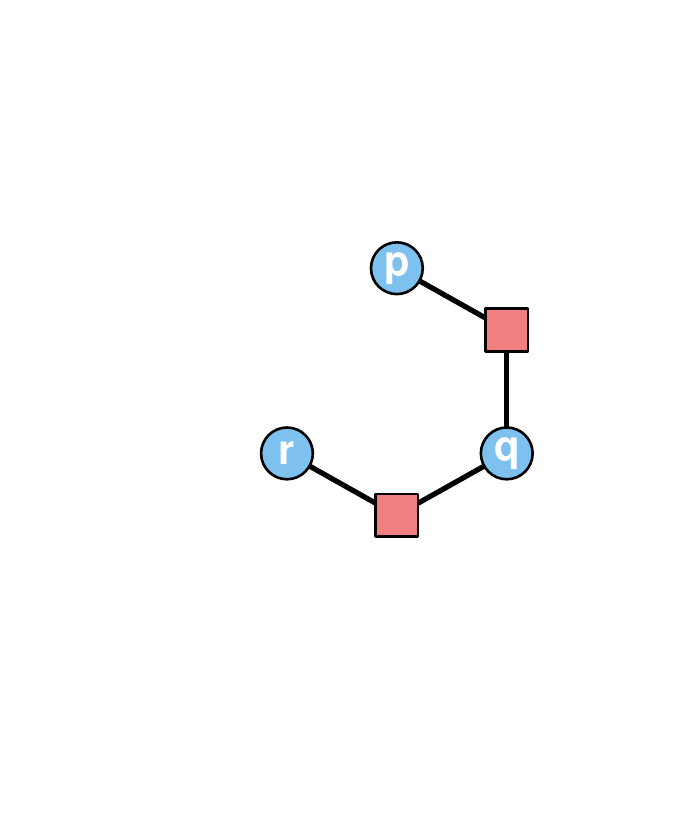}
\hspace{-3ex}a\hspace{1.5ex}\hfill
\includegraphics[width=.2\textwidth, trim = 2.5cm 2.75cm 1.25cm 2.25cm, clip = false]{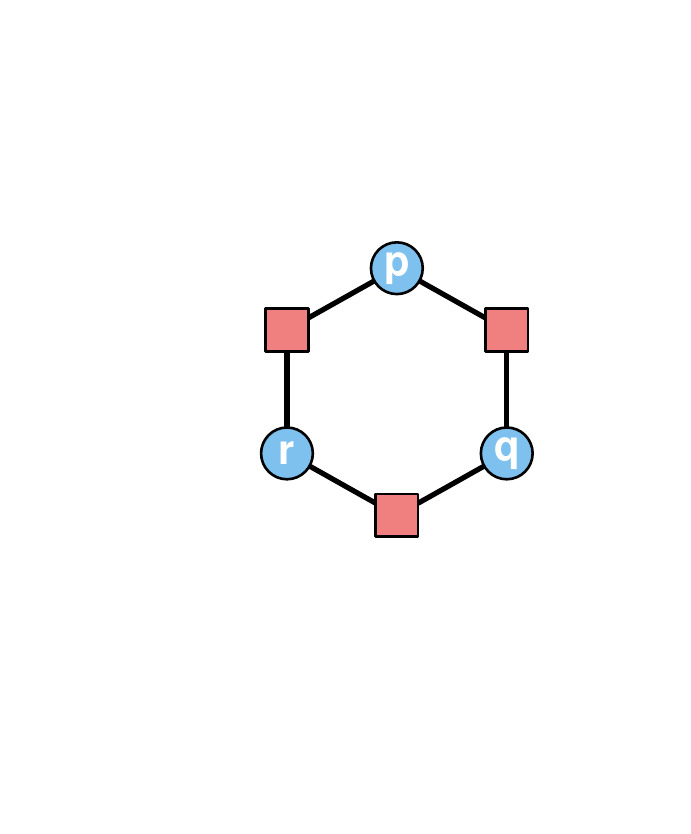}
\hspace{-3ex}b\hspace{1.5ex}\hfill
\includegraphics[width=.2\textwidth, trim = 2.5cm 2.75cm 1.25cm 2.25cm, clip = false]{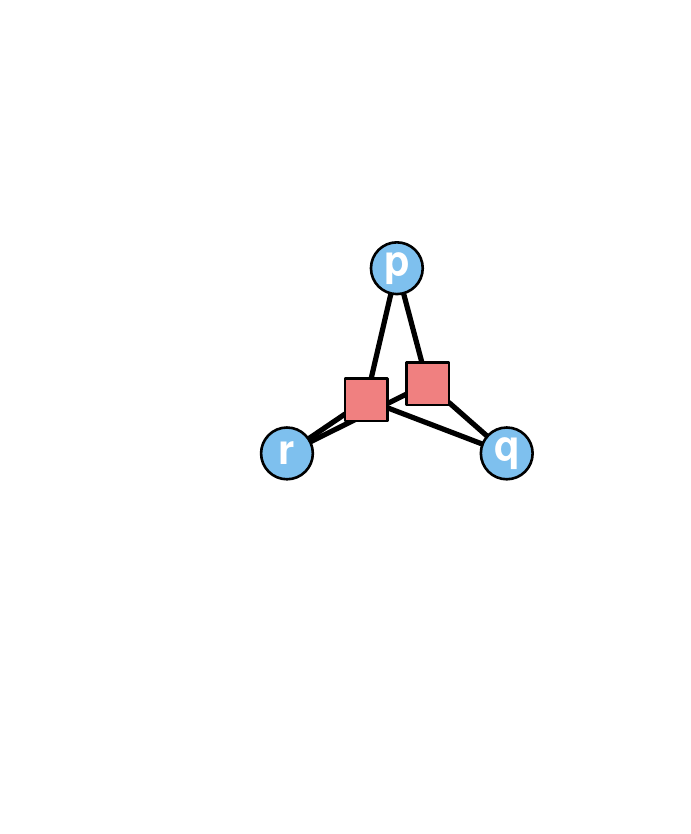}
\hspace{-3ex}c\hspace{1.5ex}\hfill
\includegraphics[width=.2\textwidth, trim = 2.5cm 2.75cm 1.25cm 2.25cm, clip = false]{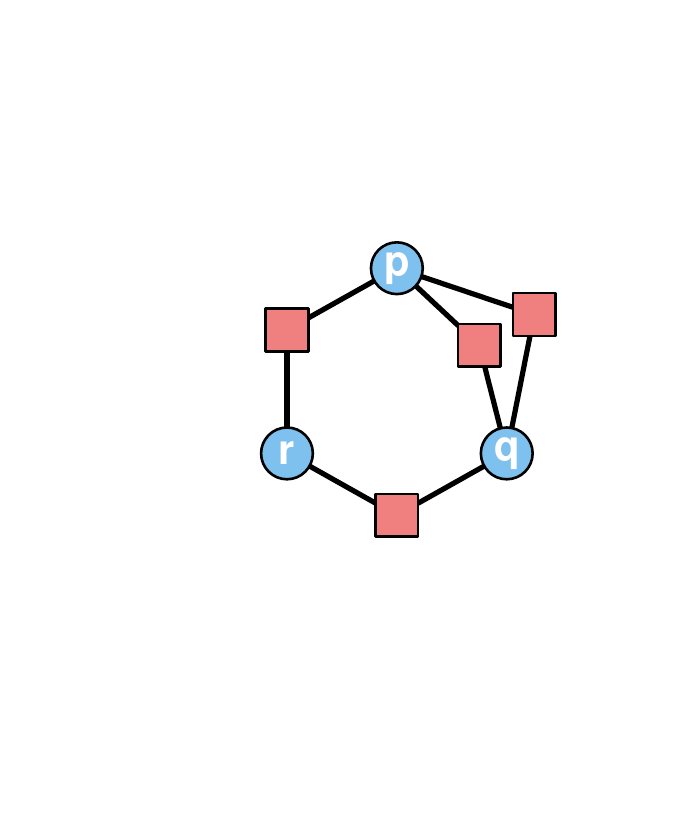}
\hspace{-3ex}d\hspace{1.5ex}
}
\caption{Four AN triads from the axiomatic analysis: (a) \(\Tr_{(1,1,0),0}\), isomorphic to \(W\); (b) \(\Tr_{(1,1,1),0}\), isomorphic to \(X\); (c) \(\Tr_{(0,0,0),2}\), from the discussion of Axiom~\ref{ax:buckle}; and (d) \(\Tr_{(2,1,1),0}\), from the discussion of Lemma~\ref{lem:pullback}.}
\label{fig:triad}
\end{figure}

The first two axioms capture important features of \(C\opsahl\). In order to prevent single events from forming closed wedges, \(C\opsahl\) is defined using only injections, from \(\widetilde\T\); Axiom~\ref{ax:induced} requires that \(\C\) include induced injections (though not all injections). In order to allow distinct events to contribute distinct wedges, \(C\opsahl\) removed the very strong congruence relation \(\approx\) imposed on the morphisms of \(C\); Axiom~\ref{ax:incongruence} allows equivalences only when events are at least structurally equivalent.

\begin{axiom}[Induced injections]\label{ax:induced}
All induced injections (hence all isomorphisms) are morphisms (i.e.\ \(\C\) contains \(\overline\T\)).
\end{axiom}

% BEGIN INJECTIVITY AXIOM OMISSION
\iffalse
\begin{axiom}[Injectivity]\label{ax:injectivity}
Morphisms are injective (on events), i.e.\ \(\C\subseteq\widetilde\T\).
\end{axiom}
\fi
% END INJECTIVITY AXIOM OMISSION

\begin{axiom}[Structural equivalence]\label{ax:incongruence}
The images of an (event) node under congruent morphisms are structurally equivalent (i.e.\ \(\sim\) is no stronger than \(\simeq\)).
%Congruent morphisms send any node (actor or event) only to structurally equivalent images.
\end{axiom}

The last two axioms address the concerns raised with \(C\opsahl\). Axiom~\ref{ax:equal} addresses the problem of weighting by admitting at most one wedge at any ordered triple. Axiom~\ref{ax:buckle} addresses the influence of bicliques by attacking their symptom: the counterintuitive way that each actor of a triad can have a wedge with none of the wedges being closed, which is not possible under \(C\). The idea is that two wedges with different centers ``hook together'' (overlap) at their shared ``side'' (pair of actors), closing each other, which is here called ``wedge buckling''. (Imagine rotating either ``open'' triad in Fig.~\ref{fig:TC} by \(120^\circ\) and overlaying it with itself.) \(C\opsahl\) violates this idea, for example at Fig.~\ref{fig:triad}c.\footnote{One could instead simply impose as an axiom the restriction of wedges and alcoves to exclusive events; Axiom~\ref{ax:buckle} provides an alternative framing for the problem.}

\begin{axiom}[Equal representation]\label{ax:equal}
At each ordered triple there exists exactly one of the following: no wedge, one open wedge, or one closed wedge.
\end{axiom}

\begin{axiom}[Wedge buckle]\label{ax:buckle}
If wedges exist at two ordered triples with different centers in a triad, then both are closed.
\end{axiom}

% BEGIN FORMULATIONS AXIOM OMISSION
\iffalse
\begin{axiom}[Congruous formulations]\label{ax:formulations}
Eqs.~\ref{eq:crate} and \ref{eq:mratio} in Def.~\ref{def:Chat} yield the same statistic.
\end{axiom}
\fi
% END FORMULATIONS AXIOM OMISSION

\paragraph{Theorems}

Three useful properties follow from certain subsets of the axioms: two triadic formulations of \(\widehat C\), which aide conceptualization and computation (Thms.~\ref{thm:census} and \ref{thm:binning}), and one characterization (Thm.~\ref{thm:existsunique}). The proofs constitute the next section.

\begin{theorem}[Census formulation]\label{thm:census}
For each triad \(\Tr_{\mu w}\), write the numbers
\begin{eqnarray*}
F_{\mu w}& =& |\Hom_{\C/\sim}(W,\Tr_{\mu w})\setminus\Hom_{\C/\sim}^X(W,\Tr_{\mu w})|
%\ \ =\ \ |\Hom_{\C/\sim}(W,\Tr_{\mu w})|-|\Hom_{\C/\sim}^X(W,\Tr_{\mu w})|
\\
S_{\mu w}& =& |\Hom_{\C/\sim}^X(W,\Tr_{\mu w})|
\end{eqnarray*}
of open and closed wedges, respectively, at \(\Tr_{\mu w}\).
If \(\widehat C\) is defined using Eq.~\ref{eq:crate}, then
\begin{equation}\label{eq:census}
\widehat{C}(G)
=\frac{\displaystyle\sum_{\mu, w}s_{\mu w}(G)S_{\mu w}}{\displaystyle\sum_{\mu, w}s_{\mu w}(G)(F_{\mu w}+S_{\mu w})}\text.
\end{equation}
\end{theorem}

Thm.~\ref{thm:census} decomposes the rate-of-closure calculation into a ratio of motifs, according to the distribution of triads in \(G\). The theorem proves useful in implementing the various global statistics, which may then be computed via arithmetic on the full census.

\begin{theorem}[Binning formulation]\label{thm:binning}
Assume Axioms~\ref{ax:induced}, \ref{ax:equal}, and \ref{ax:buckle}.
Then the triads of \(G\) can be binned into subsets \(S_\emptyset(G)\), \(S_W(G)\), and \(S_X(G)\) according as they contain none, two open, or six closed wedges; and
\begin{equation}\label{eq:binning}
\widehat{C}(G)=\frac{3|S_X(G)|}{|S_W(G)|+3|S_X(G)|}\text.
\end{equation}
\end{theorem}

Thm.~\ref{thm:binning} generalizes the simple triad census description of \(C\) in Sec.~\ref{sec:introduction}. \(C\opsahl\) does not satisfy these criteria, but \(C\excl\) does: It is recoverable from the structural triad census as \(C\excl={3\times(t_{30}+t_{31})}/{(t_{20}+t_{21}+3\times(t_{30}+t_{31}))}\).

\begin{theorem}[Existence and uniqueness]\label{thm:existsunique}
There exist unique choices of \(X\), \(W\), \(\C\), and \(\sim\) that satisfy Axioms~\ref{ax:induced}, \ref{ax:incongruence}, and \ref{ax:equal}. Moreover, these choices also satisfy Axiom~\ref{ax:buckle}. Under them, Eqs.~\ref{eq:crate} and \ref{eq:mratio} both produce \(C\excl\).
\end{theorem}

Thm.~\ref{thm:existsunique} characterizes those specializations of \(\widehat C\) that satisfy every axiom. \(C\excl\) turns out to be the unique such statistic; any alternative to \(C\excl\) still expressible in terms of Def.~\ref{def:Chat} comes at the cost of at least one axiom.
At the heart of Thm.~\ref{thm:existsunique} lies the tension between Axiom~\ref{ax:incongruence} and Axiom~\ref{ax:equal}. The former forces different types of wedges to be treated differently, and the latter allows only one of these types to figure into the formula.

\subsection{Proofs}\label{sec:proofs}

\paragraph{Triadic formulations}

A different batch of lemmas leads up to each of the second two theorems, and Thm.~\ref{thm:binning} also depends on Thm.~\ref{thm:census}. To simplify the notation, in this section let \(\Hom\) (with no subscript) denote the unspecified \(\Hom_{\C/\sim}\).

\begin{proof}[Proof of Thm.~\ref{thm:census}]
The wedges \(\Hom(W,G)\) can be partitioned according to which triad of \(G\) contains their images.
The triads of \(G\) are, in turn, partitioned by the full census.
Since the morphisms counts are fixed for isomorphic triads,
\begin{equation*}
\widehat{C}
=\frac{\displaystyle\sum_{H\subseteq G}|\Hom^X(W,H)|}{\displaystyle\sum_{H\subseteq G}|\Hom(W,H)|}
=\frac{\displaystyle\sum_{\mu,w}\Big(\sum_{\Tr_{\mu w}\cong H\subseteq G}|\Hom^X(W,\Tr_{\mu w})|\Big)}{\displaystyle\sum_{\mu,w}\Big(\sum_{\Tr_{\mu w}\cong H\subseteq G}|\Hom(W,\Tr_{\mu w})|\Big)}
=\frac{\displaystyle\sum_{\mu, w}s_{\mu w}\times S_{\mu w}}{\displaystyle\sum_{\mu, w}s_{\mu w}\times(F_{\mu w}+S_{\mu w})}\text,
\end{equation*}
where \(H\subseteq G\) ranges over the triads of \(G\).
\end{proof}

% BEGIN WEDGE-ALCOVE OMISSION
\iffalse
\begin{lemma}\label{lem:alcove}
Assume Axioms~\ref{ax:6cycle} and \ref{ax:duality}.
Then \(X=\Tr_{(1,1,1),0}\) and \(W=\Tr_{(1,1,0),0}\).
\end{lemma}

\begin{proof}
Recall that triads of only four classes are dual to triads of events (Axiom~\ref{ax:duality}).
Of these, a graph map to \(C_6\) exists only from \(\Tr_{(1,1,1),0}\), which must then be \(X\) (Axiom~\ref{ax:6cycle}).
This leaves only one choice for \(W\), up to isomorphism.
\end{proof}

Henceforth, when Axioms~\ref{ax:6cycle} and \ref{ax:duality} are assumed, these choices will be assumed without comment. Whereas the actors of a triad are \(p,q,r\), the events of \(X=\Tr_{(1,1,1),0}\) will be named \(a\) (attended by \(p\) and \(q\)), \(b\) (\(q\) and \(r\)), and \(c\); and those of \(W\) then \(a\) and \(b\).
\fi
% END WEDGE-ALCOVE OMISSION

\begin{lemma}\label{lem:exclusive}
Assume Axiom~\ref{ax:induced}.
\begin{enumerate}
\renewcommand{\theenumi}{\roman{enumi}}
\item\label{item:triples}
If \(\Tr_{\mu w}\) has an alcove, then every ordered triple of \(\Tr_{\mu w}\) has an alcove.
\item\label{item:bijection}
Given actors \(i,j,k\in G\), there is an openness-preserving bijection between the wedges of \(i,j,k\) and those of \(k,j,i\).
\end{enumerate}
\end{lemma}

Part~\ref{item:triples} follows from the symmetry of \(X\): Whatever the order of the actors, the structure of the triad is the same. Part~\ref{item:bijection} follows analogously from the more limited symmetry of \(W\), which allows \(v_0,v_1\) to be interchanged with \(v_4,v_3\) with no effect on the structure. (See Fig.~\ref{fig:triad}a,b.)

\begin{proof}
For \ref{item:triples}, pick \(\psi\in\Hom(X,\Tr_{\mu w})\) and suppose \(\psi\) takes \(v_0,v_2,v_4\) to \(i,j,k\).
Pick any permutation \(\pi\in S_3\) so that \(\pi(i,j,k)\) is an arbitrary ordered triple in \(\Tr_{\mu w}\), and let \(\rho_\pi:X\to X\) be the isomorphism taking \(v_0,v_2,v_4\) to \(\pi(v_0,v_2,v_4)\), which by Axiom~\ref{ax:induced} is a morphism.
The composition \(\psi\circ \rho_\pi:X\to\Tr_{\mu w}\) is then a morphism that takes \(v_0,v_2,v_4\) to \(\pi(i,j,k)\).

For \ref{item:bijection}, let \(\rho:W\to W\) be the isomorphism on \(W\) that exchanges \(v_0\) and \(v_4\), which is a morphism by Axiom~\ref{ax:induced}. Composition with \(\rho\) assigns any wedge \(\phi:W\to G\) that sends \(v_0,v_2,v_4\) to \(i,j,k\) to a wedge \(\phi\circ\rho\) that sends \(v_0,v_2,v_4\) to \(k,j,i\).
Moreover, since \(\rho\circ\rho\) is the identity morphism on \(W\), another composition with \(\rho\) takes \(\phi\circ\rho\) back to \((\phi\circ\rho)\circ\rho=\phi\circ(\rho\circ\rho)=\phi\).
Composition with \(\rho\) thus pairs up the wedges of the triad \(i,j,k\) centered at \(j\) (no wedge is paired with itself).
If such a wedge \(\phi\) factors through \(X\) as \(\phi=\psi\circ\iota\), then \(\phi\circ\rho\) factors through \(X\) as \(\phi\circ\rho=(\psi\circ\iota)\circ\rho=\psi\circ(\iota\circ\rho)=\psi\circ(\rho'\circ\iota)=(\psi\circ\rho')\circ\iota\), where \(\rho':X\to X\) is the isomorphism on \(X\) that exchanges \(v_0\) and \(v_4\). Thus \(\phi\) and \(\phi\circ\rho\) are open or closed together.
\end{proof}

The next two lemmas push the binning scheme of Thm.~\ref{thm:census} from triads to ordered triples. The simplicity of Eq.~\ref{eq:binning} comes from the fixed number of possible wedges (one for each ordered triple; Axiom~\ref{ax:equal}) and the symmetries between them (Lemma~\ref{lem:exclusive} and Axiom~\ref{ax:buckle}).

\begin{lemma}\label{lem:disambiguation}
Assume Axioms~\ref{ax:induced} and \ref{ax:buckle}.
Then, if a triad has two wedges with different centers, then every ordered triple in the triad has an alcove.
\end{lemma}

\begin{proof}
By Axiom~\ref{ax:buckle}, such a triad has a closed wedge, hence an alcove.
By Lemma~\ref{lem:exclusive}\ref{item:triples}, it then has an alcove at every ordered triple.
%Suppose the triad at \(i,j,k\in G\) has wedges at \((i,j,k)\) and \((j,k,i)\). By Axiom~\ref{ax:buckle}, both are closed, so there is a morphism \(\psi:X\to G\) taking \(v_0,v_2,v_4\) to \(i,j,k\). For any permutation \(\pi\in S_3\), the isomorphism \(\rho_\pi\) from the proof of Lemma~\ref{lem:alcove} composes with \(\iota\) into a graph map \(\iota\circ\rho_\pi:X\to G\) taking \(v_0,v_2,v_4\) to \(\pi(i,j,k)\), which by Axiom~\ref{ax:induced} are morphisms.
\end{proof}

\begin{lemma}\label{lem:wedgecounts}
Assume Axioms~\ref{ax:induced}, \ref{ax:equal}, and \ref{ax:buckle}.
Then each triad has exactly one of the following: no wedges, two open wedges, or six alcoves.
\end{lemma}

\begin{proof}
Each triad contains six ordered triples, which by Axiom~\ref{ax:equal} have at most one wedge each.
Lemma~\ref{lem:exclusive}\ref{item:bijection} requires that the wedges centered at any one actor either do not exist, are both open, or are both closed.
Lemma~\ref{lem:disambiguation} implies that, if two ordered triples with different centers have wedges, then all six have closed wedges. Thus the possible distributions of wedges among the six ordered triples are none, a pair of open wedges (at the same center), and six closed wedges.
\end{proof}

\begin{proof}[Proof of Thm.~\ref{thm:binning}]
Thm.~\ref{thm:census} provides Eq.~\ref{eq:census}, which respects triad classes.
Lemma~\ref{lem:wedgecounts} implies that either \(S_{\mu w}=F_{\mu w}=0\), \(S_{\mu w}=0\) and \(F_{\mu w}=1\), or \(S_{\mu w}=3\) and \(F_{\mu w}=0\) for every triad class.
Binning these classes into \(S_\emptyset\), \(S_W\), and \(S_X\), respectively, achieves the result.
\end{proof}

\paragraph{Characterization}

The characterization theorem takes place over three steps: First, the three assumed axioms only allow wedges and alcoves with no inclusive events (\(\overline\T\)). (This makes Axiom~\ref{ax:buckle} unnecessary.) Second, the equal representation of Axiom~\ref{ax:equal} requires that any wedges at the same ordered triple of actors are congruent (\(\approx\)), but when inclusive events are ignored the weaker relation \(\simeq\) is enough. This limits the options to the two formulations in Def.~\ref{def:Chat} under the category \(\T/\simeq\). Third, these formulations agree under certain conditions, which turn out to be satisfied under \(\T/\simeq\).

\begin{lemma}\label{lem:induced}
Assume Axioms~\ref{ax:induced}, \ref{ax:incongruence}, and \ref{ax:equal}.
Then any wedge or alcove is an induced injection.
\end{lemma}

\begin{proof}
The only way for a wedge or alcove to not be an induced injection is for it to send some event to an inclusive event.
Suppose the alcove \(\psi:X\to G\) sends \(v_0,v_1,v_2,v_3,v_4,v_5\) to \(i,d,j,e,k,f\), where at least one of the events \(d,e,f\) is inclusive to the triad at \(i,j,k\). (\(d\), \(e\), and \(f\) need not be distinct.) If \(d\) or \(e\) is inclusive, then \(\psi\circ\iota:W\to G\) is a wedge with an inclusive event. If only \(f\) is inclusive, then let \(\rho:X\to X\) be the isomorphism sending \(v_0,v_1,v_2,v_3,v_4,v_5\) to \(v_2,v_3,v_4,v_5,v_0,v_1\), so that the composition \(\psi\circ\rho\circ\iota:W\to G\) sends \(v_0,v_1,v_2,v_3,v_4\) to \(j,e,k,f,i\). By Axiom~\ref{ax:induced}, \(\psi\circ\rho\circ\iota\) is a wedge with an inclusive event. It is enough, therefore, to prove the result for wedges.

\begin{figure}[h]
\centerline{
\hfill
\includegraphics[width=.2\textwidth, trim = 2.5cm 2.75cm 1.25cm 2.25cm, clip = false]{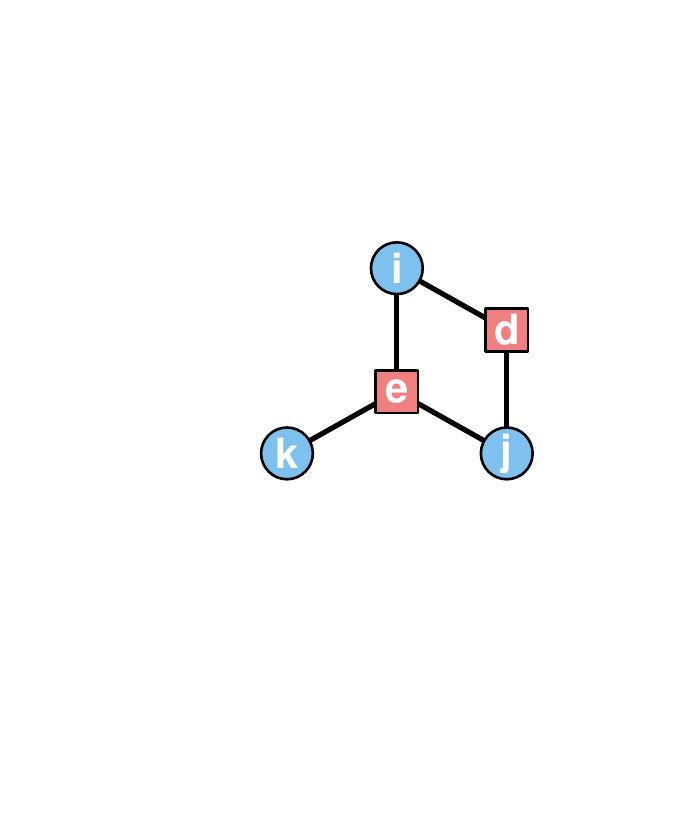}
\hspace{-3ex}a\hspace{1.5ex}\hfill
\includegraphics[width=.2\textwidth, trim = 2.5cm 2.75cm 1.25cm 2.25cm, clip = false]{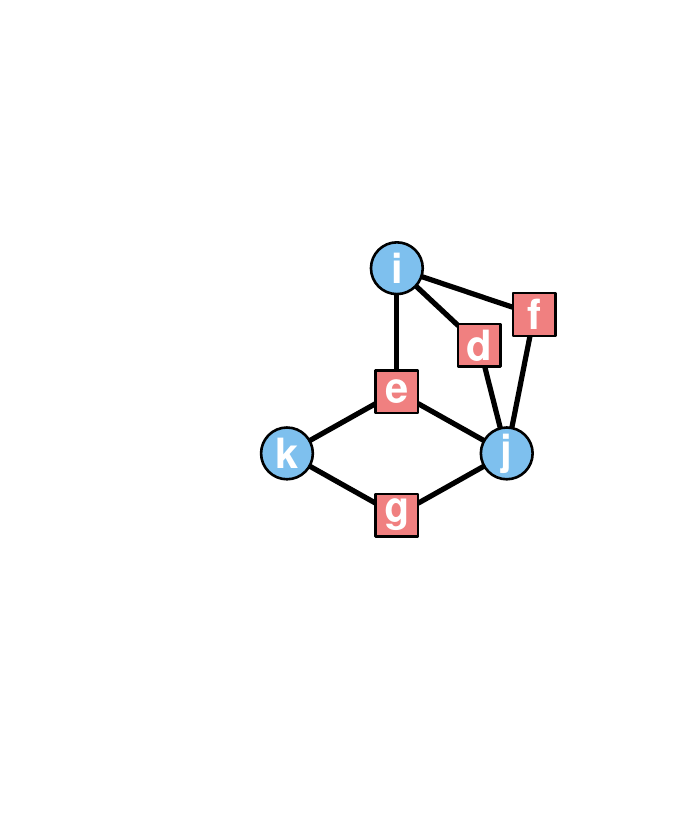}
\hspace{-3ex}b\hspace{1.5ex}\hfill
}
\caption{From the proof of Lemma~\ref{lem:induced}: (a) the image of \(\psi:W\to G\) and (b) the necessary subgraph of \(G\) containing (a).}
\label{fig:induced}
\end{figure}

So suppose the wedge \(\phi:W\to G\) sends \(v_0,v_1,v_2,v_3,v_4\) to \(i,d,j,e,k\), where at least one of \(d\) and \(e\) is inclusive to the triad at \(i,j,k\). Obtain \(G'\) from \(G\) by adding events \(f\), attended only by \(i\) and \(j\), and \(g\), attended only by \(j\) and \(k\). (See Fig.~\ref{fig:induced}.) The subgraph inclusion \(\sigma:G\to G'\) is an induced injection, hence by Axiom~\ref{ax:induced} a morphism. Then the composition \(\sigma\circ\phi:W\to G'\) is a wedge. The graph map \(\phi':W\to G'\) sending \(v_0,v_1,v_2,v_3,v_4\) to \(i,f,j,g,k\) is an induced injection since \(f\) and \(g\) are exclusive events, so by Axiom~\ref{ax:induced} \(\phi'\) is also a wedge---at the same ordered triple as \(\sigma\circ\phi\). Axiom~\ref{ax:incongruence} implies that these wedges are incongruent, which contradicts Axiom~\ref{ax:equal}. Thus \(\phi\) cannot exist.
\end{proof}

\begin{lemma}\label{lem:congruent}
Assume Axiom~\ref{ax:equal}.
Then \(\sim\) is at least as strong as \(\approx\) on the wedges and alcoves.
\end{lemma}

\begin{proof}
The claim is that any two wedges or alcoves on the same ordered triple of actors are congruent.
If they were not, then Axiom~\ref{ax:equal} would be violated.
\end{proof}

The {\df pullback} \(\iota^\ast:\Hom(X,G)\to\Hom(W,G)\) sends any alcove \(\psi\in\Hom(X,G)\) to the wedge \(\psi\circ\iota:W\to G\).
To understand Lemma~\ref{lem:pullback}, note that the image of \(\iota^\ast\) is in \(\Hom^X(W,G)\)---that is, each such \(\psi\circ\iota\) factors through \(X\) (via the morphism \(\psi\) began with).

\begin{lemma}\label{lem:pullback}
Eqs.~\ref{eq:crate} and \ref{eq:mratio} yield the same statistic if and only if \(\iota^\ast\) is injective.
\end{lemma}

This lemma is not satisfied, for instance, by the category \(\widetilde\T/=\) underlying \(C\opsahl\): The wedge \(\phi:W\to\Tr_{(2,1,1),0}\) (Fig.~\ref{fig:triad}d) sending \(v_0,v_2,v_4\) to \(v_2,v_4,v_0\) can be closed by either of the events shared by \(v_0\) and \(v_2\). \(C\opsahl\), defined using Eq.~\ref{eq:crate}, counts this as one closed wedge. Its counterpart \(\widehat C\), defined using Eq.~\ref{eq:mratio}, however, counts two alcoves, one for each choice of event---that is, \(\phi\) factors through \(X\) in two ways. (Under this statistic, in fact, \(\widehat C(\Tr_{(2,1,1),0})=\frac{6}{5}\).)

\begin{proof}
Given \(\phi\in\Hom^X(W,G)\), by definition there exists \(\psi\in\Hom(X,G)\) such that \(\phi=\psi\circ\iota\); thus, in any case, \(\iota^\ast\) has image \(\Hom^X(W,G)\).
The second condition therefore amounts to \(\iota^\ast\) being a bijective correspondence between its domain \(\Hom(X,G)\) and its range \(\Hom^X(W,G)\).
Since \(\iota^\ast\) is surjective and its domain and range are finite, this is true if and only if the domain and range have equal size.
Since the denominators of Eqs.~\ref{eq:crate} and \ref{eq:mratio} are equal, this is true if and only if the formulations are equal, unless both are undefined.
This occurs only when \(\Hom(W,G)\) is empty, in which case both \(\Hom(X,G)\) and \(\Hom^X(W,G)\) are also empty.
% BEGIN TWO-WAY PROOF OMISSION
\iffalse
First assume the two formulations are equal.
This means that \(\Hom_{\C/\sim}(X,G)\) contains the same number of morphisms as \(\Hom_{\C/\sim}^X(W,G)\) (zero if \(\Hom_{\C/\sim}(W,G)\) is empty).
Given \(\phi\in\Hom_{\C/\sim}^X(W,G)\), by definition there exists \(\psi\in\Hom_{\C/\sim}(X,G)\) such that \(\phi=\psi\circ\iota\); thus \(\iota^\ast\) is surjective.
Since both sets of morphisms are finite, this makes \(\iota^\ast\) injective, hence bijective.

Now assume that \(\iota^\ast\) is injective.
Again by 
\fi
% END TWO-WAY PROOF OMISSION
\end{proof}

\begin{proof}[Proof of Thm.~\ref{thm:existsunique}]
%Axiom~\ref{ax:injectivity} implies that wedges and alcoves are injective.
Lemma~\ref{lem:induced} implies that wedges and alcoves are induced injections.
By Axiom~\ref{ax:induced}, all of these are morphisms.
As far as Def.~\ref{def:Chat} is concerned, then, \(\C\) is \(\overline{\T}\).

Lemma~\ref{lem:congruent} implies that the congruence relation \(\sim\) is no weaker than \(\approx\).
Since the events of two wedges or alcoves at the same ordered triple must be exclusive, hence structurally equivalent in the triad, the relations \(\simeq\) and \(\approx\) have the same effect in this case; \(\C/\sim\) is \(\overline\T/\simeq\).
This establishes uniqueness.

For the auxiliary claim, suppose \(\psi,\psi'\in\Hom_{\overline\T/\simeq}(X,G)\) are incongruent.
By the choice of \(\overline\T\), their respective images of \(v_1,v_3,v_5\) must be exclusive.
If \(\psi,\psi'\) agree on all three actors, then, by the choice of \(\simeq\), they are congruent.
So \(\psi,\psi'\) must disagree on some actor; say \(\psi(v_0)\neq\psi'(v_0)\).
This implies that \(\psi\circ\iota(v_0)=\psi(v_0)\neq\psi'(v_0)=\psi'\circ\iota(v_0)\), hence that \(\iota^\ast(\psi)\neq\iota^\ast(\psi')\).
Thus, \(\iota^\ast\) is injective.
By Lemma~\ref{lem:pullback}, both formulations of Def.~\ref{def:Chat} produce the same statistic.

It remains to verify that \(C\excl\) actually satisfies each axiom; this is left to the reader.
\end{proof}

\section{Empirical analyses}\label{sec:empirical}

This section applies triadic tools, including \(C\), \(C\opsahl\), and \(C\excl\), to three empirical networks. Sec.~\ref{sec:instrument} assesses the clustering coefficients as measurement instruments, by comparing their performances on the empirical networks. The assessments consider reliability, validity, redundancy, and practicality, and are illustrated in two case studies. Sec.~\ref{sec:properties} performs triadic analyses of the empirical networks, using the census and the clustering coefficients. The analyses draw upon and extend concepts from previous studies (see Sec.~\ref{sec:background}), including strong triadic closure, brokerage, and influence.

\subsection{Instrumentation}\label{sec:instrument}

\paragraph{Data}

The analyses employ three empirical networks: The social activity attendance network DG1 comes from another table in the same study as above \cite{dgg-deep}, and has seen extensive use as a test case for node classification and community detection techniques \cite{f-finding}. A subset of interlocking directorates data, from a study of corporate philanthropy in Minneapolis--St. Paul \cite{g-social,wf-social}, constitute GWF. Finally, MR refers to the collaboration network constructed from the {\it Mathematical Reviews} bibliographic database, which is maintained by the American Mathematical Society, over the years 1985--2008. These networks are constructed from a range of types and volumes of social interaction data and have appeared in previous studies that provide checks and comparisons for the present work. Two (DG1 and MR) have time-labeled events.\footnote{DG1 is assumed to consist of events spanning nine months \cite{f-finding}; however, whereas the study took place over two years, other orderings are not impossible.}

\begin{table}[h]
  \caption{Structural censuses of DG1, GWF, and two intervals of MR. The column indicates the presence (1) or absence (0) of an inclusive event; the row indicates the number of pairs of actors who attend at least one exclusive event.}
  \label{tab:structural}
  \begin{minipage}{\textwidth}
  % latex table generated in R 3.1.0 by xtable 1.7-4 package
% Sun May 24 18:23:38 2015
\begin{tabular}{c|rr|rr|rr|rr|}
  \hline\hline 
 & \multicolumn{2}{c|}{DG1} & \multicolumn{2}{c|}{GWF} & \multicolumn{2}{c|}{MR (1985-7)} & \multicolumn{2}{c|}{MR (2005-7)} \\   
 & 0 & 1 & 0 & 1 & 0 & 1 & 0 & 1 \\ 
  \hline
0 & \(0\) & \(17\) & \(0\) & \(284\) & \(80,747,526,018,836\) & \(17,275\) & \(725,892,036,097,769\) & \(76,558\) \\ 
  1 & \(39\) & \(240\) & \(266\) & \(886\) & \(4,721,138,210\) & \(8,611\) & \(38,496,757,064\) & \(51,599\) \\ 
  2 & \(146\) & \(253\) & \(452\) & \(521\) & \(133,630\) & \(2,014\) & \(909,505\) & \(15,185\) \\ 
  3 & \(45\) & \(76\) & \(130\) & \(61\) & \(886\) & \(129\) & \(5,585\) & \(1,055\) \\ 
   \hline
\hline
\end{tabular}

  \end{minipage}
\end{table}

Table~\ref{tab:structural} presents the structural censuses of the networks. The higher-order structure lost in projection lives mostly in the second column of each census. Several differences between DG1 and GWF, on one hand, and MR, on the other, are apparent: MR is far larger, with triads concentrated among the less-connected; ``symmetric exclusive'' triads (\(t_{30}\), see Ex.~\ref{ex:dg2}) make up a minuscule fraction, undercut only by that of ``symmetric complete'' triads (\(t_{31}\)). In contrast, DG1 and GWF have remarkably similar profiles: the event-free triads number \(t_{00}=0\), and the largest tallies occupy a northeast--southwest diagonal band away from the least and most connected types. This indicates that the smaller networks are more uniformly connected, with fewer poorly-connected actors. This difference likely reflects non-uniformity in the coverage of researchers in MR \cite{lc-community}, e.g.\ as equally prolific researchers on the periphery of mathematics appear less frequently in MR \cite{bfmnrfil-evolutionary}.

The editors assign to each publication one primary and any number of secondary Mathematical Subject Classification (MSC) codes from a hierarchical scheme. At the coarsest level, publications are binned into 64 groups (for instance, algebraic geometry, partial differential equations, and astronomy and astrophysics). For the assessments, 64 subnetworks are constructed by partitioning the literature by primary classification. Of these, 39 satisfy the following inclusion criteria over each adjacent 3-year interval from 1985--7 to 2006--8: the literature is not empty; each of \(C\), \(C\opsahl\), \(C\excl\), and \(D\) is defined; and no two of these statistics are simultaneously zero. Since their curation and construction are systematic, differences in structure among these networks should only reflect differences in the cultures of research publication and limitations of MR coverage. (Nonetheless, size and density are known to influence measures of TC.)

\paragraph{Criteria}

While the statistics surveyed in Sec.~\ref{sec:exclusive} are hopefully intuitive, it is not yet clear that they are useful instruments.\footnote{Strictly speaking, the ``instrument'' that assigns a clustering coefficient to a social network includes the collection of sociometric data and the construction of the bipartite graph as well as the graph-theoretic calculation and the device that performs it; only the calculation is meant here.} This section assesses the local and global definitions of \(C\), \(C\opsahl\), and \(C\excl\) on the basis of stability, concurrent validity, discriminability (meant to reflect practicality), and distinguishability (non-redundancy). The assessments are performed on three samples: the 18 actors of DG1, the 26 actors of GWF, and the 39 disciplines of MR (along adjacent 3-year intervals).
The criteria are conceptualized and assessed as follows:
\begin{itemize}
\item
An instrument is {\df stable} if it yields similar measurements of the same subject at different times. Stability is assessed, on pairs of values at the same MR discipline at adjacent intervals, as the proportion \(\frac{SSM}{SST}\) of the variation in the values accounted for by the pairing in a one-way analysis of variance \cite{ab-measurement}.
%(Adjacent intervals minimize the effect of network evolution.)
\item
Both \(C\opsahl\) and \(C\excl\) are hypothesized to measure properties of graphs that can also be measured in other ways: As mentioned in Sec.~\ref{sec:exclusive}, an alternative correction to \(C\) for event size in ANs is the quotient of \(C\) by its expected value \(C_{\rm rand}\) on an equivalent random bipartite graph.\footnote{Here \(C_{\rm rand}\) is calculated two ways: For the smaller networks DG1 and GWF, take the mean (local) values of \(C\) across 1000 randomly generated bipartite graphs having the same actor and event degree sequences \cite{cdhl-sequential,ah-networksis}. For the MR subnetworks, use the asymptotic approximation \cite{nsw-random}.} Sec.~\ref{sec:exclusive} also suggested that \(C\excl\) may measure dynamic TC, defined as \(D\). The {\df concurrent validity} of each measure shall be assessed as its coefficient of determination \(R^2\) with its alternative \cite{kw-validity}.
\item
Two instruments designed to measure distinct properties shall be called {\df distinguishable} if they yield divergent values on the same subjects. Whereas the coefficient of determination between these values gives their concurrent validity, the remaining proportion of variance, \(1-R^2\), shall assess their distinguishability.
\item
An instrument is {\df discriminable} if its values in practice are dispersed throughout its theoretical range \cite{csc-framework}. (Sec.~\ref{sec:introduction} criticized \(C\) for having low discriminability on ANs.) Discriminability is assessed as the variance \(s^2\) of an instrument's values for a sample of subjects; the standardized values \(4s^2\) are reported, so that discriminability theoretically ranges from \(0\) (all values equal; statistic is useless) to \(1\) (values evenly split between \(0\) and \(1\); statistic perfectly dichotomizes the subjects). A statistic whose values follow a Gaussian distribution centered at \(0.5\) with standard deviation \(0.25\) (and cut off at the \(95\%\) thresholds) has discriminability just under \(\frac{1}{4}\), while one whose values are uniformly distributed has discriminability \(\frac{2}{3}\).
\end{itemize}

On MR, each assessment is performed on the pooled values across all intervals. For instance, each statistic's stability is computed on \(39\times 7=273\) ordered pairs of values.

\paragraph{Results}

\begin{table}[h]
  \caption{Evaluations of three clustering coefficients taken over actors (DG1 and GWF) or subnetworks (adjacent 3-year intervals of MR).}
  \label{tab:test}
  \begin{minipage}{\textwidth}
    \begin{tabular}{l|rrr|rrr|rrr|}
      \hline
      \hline
      &  \multicolumn{3}{c|}{Classical} & \multicolumn{3}{c|}{Opsahl} & \multicolumn{3}{c|}{Exclusive}  \\
&  DG1 & GWF & MR & DG1 & GWF & MR & DG1 & GWF & MR  \\\hline
Stability & \(\) & \(\) & \(0.781\) & \(\) & \(\) & \(0.403\) & \(\) & \(\) & \(0.457\) \\
Validity & \(\) & \(\) & \(\) & \(0.622\) & \(0.296\) & \(0.113\) & \(0.058\) & \(\) & \(0.399\) \\
Dist. (Classical) & \(\) & \(\) & \(\) & \(0.950\) & \(0.940\) & \(0.999\) & \(0.492\) & \(0.732\) & \(0.924\) \\
Dist. (Opsahl) & \(\) & \(\) & \(\) & \(\) & \(\) & \(\) & \(0.915\) & \(0.592\) & \(0.948\) \\
Discriminability & \(0.005\) & \(0.013\) & \(0.047\) & \(0.051\) & \(0.050\) & \(0.026\) & \(0.205\) & \(0.224\) & \(0.001\) \\

      \hline
      \hline
    \end{tabular}
  \end{minipage}
\end{table}

The test results constitute Table~\ref{tab:test}. (Non-meaningful or redundant cells are left empty. Plots for each assessment are included in the supplement.) \(C\) is by far the most stable of the statistics (\(\frac{SSM}{SST}=0.78\)), with less than half of the variation in \(C\opsahl\) and \(C\excl\) each interval accounted for by the previous.
Tests of validity were inconsistent. \(C\opsahl\) was highly correlated with \(C/C_{\rm rand}\) across the women of DG1, but much less so across the CEOs of GWF and the disciplines of MR. Conversely, \(C\excl\) accounted for \(40\%\) of the variance in \(D\) across the disciplines but none across the women. Some heteroskedasticity is also visible in the plots of \(C\excl\). There is strong evidence here that these instruments are closely related, but only in certain limited settings.

The three statistics are highly distinguishable; at worst, \(C\) explains half of the variance in \(C\excl\) across the women of DG1 (\(1-R^2=0.49\)). This, residual plots reveal, is due to a consistent negative relationship.
\(C\) and \(C\opsahl\) are poor discriminants, but on the actors of the smaller networks \(C\excl\) takes values nearly as distributed over \([0,1]\) as the hypothetical cut-off Gaussian. This makes sense in light of the higher {\em average} rates of TC in DG1 and GWF; by comparison, the many highly-connected triads of MR are overwhelmed by the more partially-connected, which \(C\) is better-equipped to discriminate among (and does).  Overall, the assessments lend some legitimacy to the uses of \(C\), \(C\opsahl\), and \(C\excl\) in the next section, but more persuasive assessments of single-value network statistics would be helpful.

\begin{example}\label{ex:DG1}
\begin{table}
  \caption{Measures of local triadic closure and centrality in DG1.}
  \label{tab:dg1local}
  \begin{minipage}{\textwidth}
  % latex table generated in R 3.1.0 by xtable 1.7-4 package
% Tue Jun  2 14:48:33 2015
\begin{tabular}{lrrrrrrr}
  \hline
\hline
 & Classical & Opsahl & Exclusive & Dynamic & TwoWalk & Eigenvector & TwoWalkCorrected \\ 
  \hline
Evelyn & 0.897 & 0.767 & 0.448 & 0.576 & 0.319 & 0.335 & 0.015 \\ 
  Laura & 0.962 & 0.842 & 0.487 & 0.692 & 0.286 & 0.309 & 0.023 \\ 
  Theresa & 0.897 & 0.752 & 0.145 & 0.650 & 0.358 & 0.371 & 0.013 \\ 
  Brenda & 0.962 & 0.839 & 0.450 & 0.692 & 0.292 & 0.313 & 0.021 \\ 
  Charlotte & 1.000 & 1.000 & 1.000 & 1.000 & 0.154 & 0.168 & 0.014 \\ 
  Frances & 0.962 & 0.869 & 0.778 & 0.000 & 0.198 & 0.209 & 0.011 \\ 
  Eleanor & 0.962 & 0.796 & 0.531 & 0.692 & 0.220 & 0.228 & 0.008 \\ 
  Pearl & 0.933 & 0.646 & 0.467 & 0.636 & 0.187 & 0.180 & -0.007 \\ 
  Ruth & 0.897 & 0.670 & 0.328 & 0.650 & 0.242 & 0.236 & -0.006 \\ 
  Verne & 0.897 & 0.674 & 0.393 & 0.576 & 0.231 & 0.218 & -0.013 \\ 
  Myra & 0.933 & 0.714 & 0.556 & 0.273 & 0.204 & 0.187 & -0.017 \\ 
  Katherine & 0.933 & 0.770 & 0.536 & 0.273 & 0.237 & 0.220 & -0.017 \\ 
  Sylvia & 0.897 & 0.746 & 0.300 & 0.576 & 0.292 & 0.277 & -0.015 \\ 
  Nora & 0.897 & 0.838 & 0.663 & 0.725 & 0.281 & 0.264 & -0.017 \\ 
  Helen & 0.897 & 0.816 & 0.661 & 0.611 & 0.215 & 0.201 & -0.014 \\ 
  Dorothy & 0.933 & 0.541 & 0.467 & 0.000 & 0.143 & 0.131 & -0.012 \\ 
  Olivia & 1.000 & 0.581 & 1.000 & 1.000 & 0.088 & 0.070 & -0.019 \\ 
   \hline
\hline
\end{tabular}

  \end{minipage}
\end{table}

Consider the TC of the women who constitute DG1 (Table~\ref{tab:dg1local}, with structural equivalents Olivia and Flora represented by Olivia. Centrality scores will be used in Sec.~\ref{sec:properties}. The supplement contains the table for GWF). Partitioning and core--periphery algorithms tend to identify Pearl, Ruth, and Verne as intergroup bridges or peripheral group members \cite{f-finding}, though in terms of classical TC their neighborhoods are unremarkable. In contrast, these women exhibit the lowest Opsahl TC of the group, and two (Ruth and Verne) are among the three with lowest exclusive TC. These observations attest to the greater discriminability of these statistics.

Pearl, however, has exclusive TC on par with several women in the cores of the two communities (Evelyn, Laura, and Dorothy). Theresa and Sylvia, on the other hand---who are usually placed near the cores of their respective groups within DG1, rather than toward the periphery with Ruth and Verne---show lower exclusive TC. This is due to the high number of events (8 and 7) these women attended. It may be that the study window omitted events attended by their neighbors in their absence, though both women attended events as early as March and as late as September, making this less likely; or it may be that these women played distinctive networking roles in their respective groups, to which traditional algorithms are not sensitive (see Sec.~\ref{sec:properties}).
\end{example}

\begin{example}\label{ex:MR}
A previous study of MR \cite{bfmnrfil-evolutionary} compared two subnetworks, constructed via a nearly even partition of primary MSCs into ``pure'' and ``applied''.\footnote{The partition is coarse and provisional, but reveals a real difference between the research cultures; these subnetworks displayed consistently and characteristically different behavior.} The analysis of TC used \(C\) and \(C/C_{\rm rand}\); the time series are reproduced in Fig.~\ref{fig:win3C} (``Classical'' and ``BipartiteCorrected''). While \(C\) revealed persistent properties of MR, e.g.\ that the applied research community saw more classical TC than the pure, \(C/C_{\rm rand}\) revealed discordant trends in pure and applied research. Both statistics arguably discriminated well, and certainly they were distinguishable from each other.

\begin{figure}
\centerline{\includegraphics[width=\textwidth]{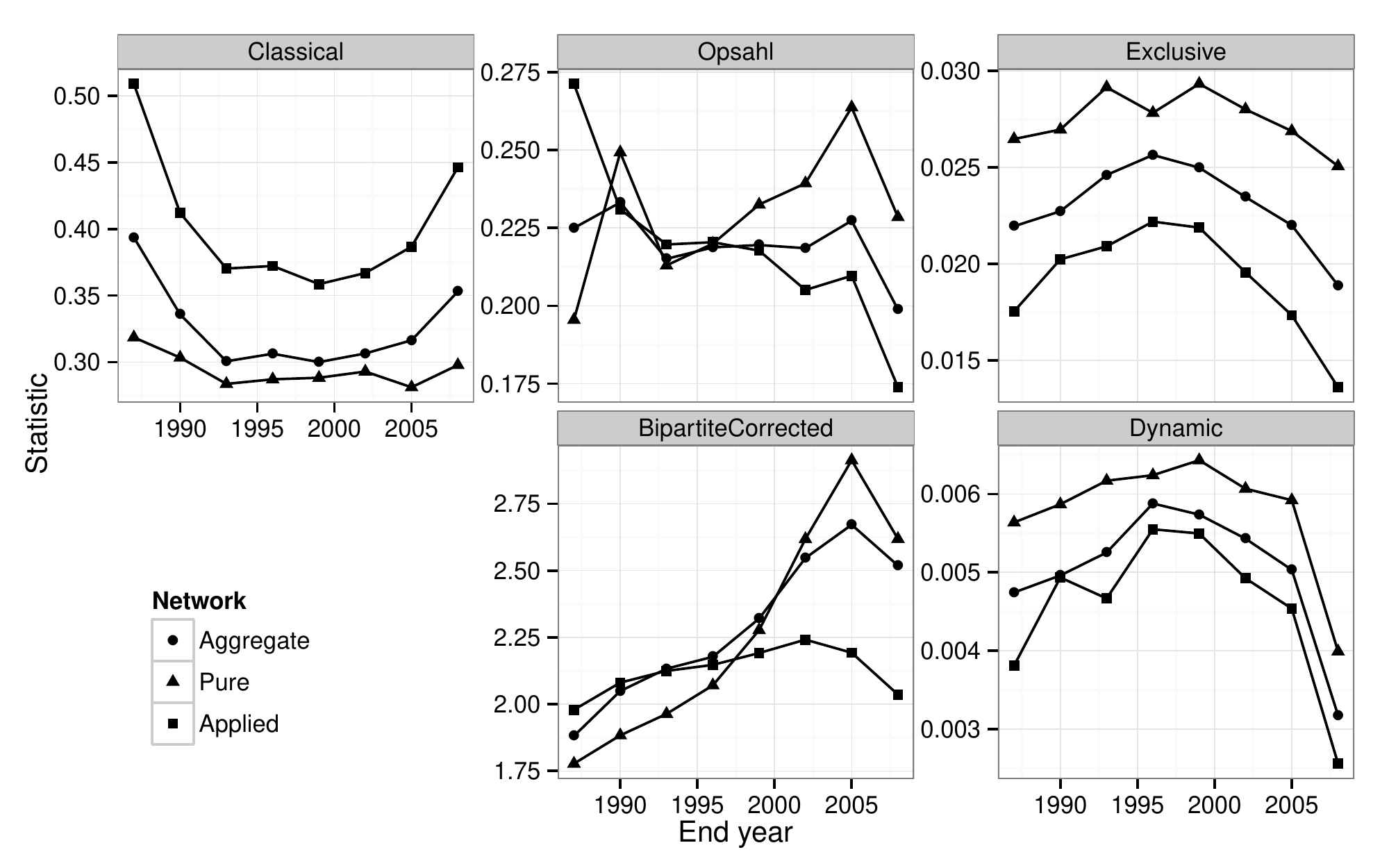}}
\caption{Three global clustering coefficients and alternative measures for two, on the aggregate, pure, and applied MR networks along adjacent 3-year intervals.}
\label{fig:win3C}
\end{figure}

Fig.~\ref{fig:win3C} also includes time series for \(C\opsahl\) and \(C\excl\). The three trajectories of \(C\opsahl\) mimic those of \(C/C_{\rm rand}\) up to a linear transformation; the rates of change are clearly least in the pure network and greatest in the applied. More impressive is the stark resemblance between \(C\excl\) and \(D\), up to scale. \(C\opsahl\) and \(C\excl\) both are less discriminating than \(C\) in absolute terms, though all three are clearly distinguishable. Like \(C\), \(C\excl\) measures a persistent difference between the research cultures: Pure research is better-characterized by exclusive (or dynamic) TC than applied. The negative relationship between \(C\) and \(C\excl\) is evident here: the relative values of \(C\excl\) are inverted from those of \(C\), both in the ordering of the networks and in the concavity of the trends.
\end{example}

\subsection{Triadic closure in affiliation networks}\label{sec:properties}

\paragraph{Strong triadic closure}\label{sec:stc}

In social networks with ties of different strengths, the STC hypothesis predicts that, when two pairs of actors in a triad are {\em strongly} tied, then the third pair will tend to be at least {\em weakly} tied \cite{g-strength}. Investigators have formalized and tested this principle in a variety of ways, often in terms of the frequency, duration, or intimacy of relations, or of the proportion of relations above some threshold of strength \cite{f-sociological}. One conversion approach to STC in ANs is therefore to apply these methods to a weighted projection.

The full triad census offers a direct approach: Within a triad, it makes sense to infer stronger ties between actors from exclusive events than from inclusive events, consistent with the principle that higher-attendance events foster weaker pairwise connections \cite{ge-measuring}.
Accordingly, take the {\df wedge strength} of the ordered triple \((i,j,k)\) to be the number of 4-paths along exclusive events from \(i\) through \(j\) to \(k\), and take \(i\) and \(k\) to be {\df (at least) weakly tied} if there is any 2-path between them.
Thus, the triple \((p,q,r)\) in the triad \(\Tr_{\mu w}\) have wedge strength \(\mu_1\times\mu_2\) and are weakly tied if \(\mu_3+w>0\).
STC shall be measured in an AN as the probability of a weak tie conditional on wedge strength.\footnote{An alternative measure is the expected number of events attended by \(i\) and \(k\), conditioned on the wedge strength of \((i,j,k)\). The results in MR, not reported, are similar to those shown.}

\begin{figure}
\centerline{\includegraphics[width=\textwidth]{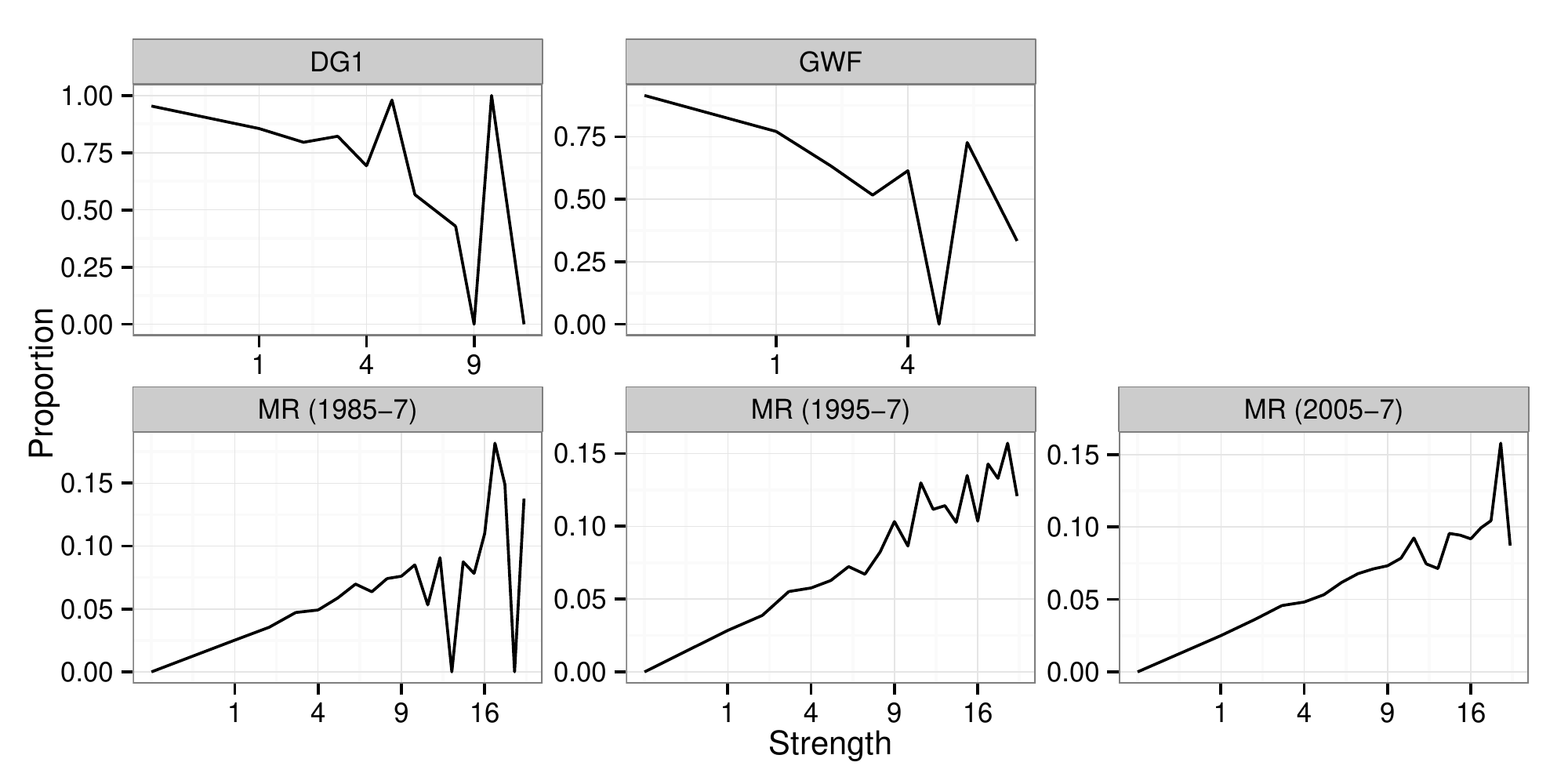}}
\caption[Strong triadic closure]{Conditional probability \(\Pr(\mu_3+w>0\mid \mu_1\times\mu_2=s)\) of a weak tie versus tie strength \(s\), up to \(s={20}\). (Note the square root--scale horizontal axis.)}
\label{fig:stc}
\end{figure}

Fig.~\ref{fig:stc} presents the conditional probabilities for DG1, GWF, and MR over three evenly-spaced 3-year intervals, using a square-root scale on the horizontal axis.
In DG1 and GWF, increasing wedge strength is associated (albeit noisily) with a lower rate of weak tie formation, in defiance of STC.
In contrast, STC in MR is well-modeled by the proportionality
\begin{equation}\label{eq:stc}
\Pr(\mu_3+w>0\mid \mu_1\times\mu_2=s)\ \propto\ s^{\frac{1}{2}}\text.
\end{equation}
Furthermore, though STC makes no predictions about the proportion of ties between actors who have no neighbors in common (the case \(s=0\)), in MR this case is accurately extrapolated from the pattern across wedges of positive strength.

\paragraph{Connectedness and constraint}\label{sec:constraint}

The STC hypothesis is intimately tied to the study of brokerage, in that connections among an actor \(i\)'s neighbors can be thought to constrain \(i\)'s potential to broker between them \cite{b-structural}. Constraint is formulated as a product of \(i\)'s investment in connecting with their neighbors and the connectedness of these neighbors with each other. The local clustering coefficient provides a simple model of constraint: If \(i\) has \(d\) neighbors, each \(j\) of whom is tied to \(d(j)\) of \(i\)'s other neighbors, then the constraint on \(i\) due to \(j\) can be defined as
\begin{equation*}\label{eq:constraint1}
c(i,j)=\frac{1}{d}\times\frac{d(j)}{d-1}=\frac{d(j)}{d(d-1)}\text,
\end{equation*}
with total constraint \(c(i)=\sum_jc(i,j)=C(i)\). The equivalent formulation
\begin{equation}\label{eq:constraint2}
c(i,j)=\frac{|\{\text{wedges at \(i\) w/ \(j\)}\}|}{|\{\text{wedges at \(i\)}\}|}
\times\frac{|\{\text{closed wedges at \(i\) w/ \(j\)}\}|}{|\{\text{wedges at \(i\) w/ \(j\)}\}|}
=\frac{|\{\text{closed wedges at \(i\) w/ \(j\)}\}|}{|\{\text{wedges at \(i\)}\}|}
\end{equation}
generalizes neatly to the terms of Def.~\ref{def:Chat}. Thus the family of local clustering coefficients may be viewed as a family of alternative measures of constraint in ANs.\footnote{This should be compared cautiously to previous approaches that conditioned bipartite clustering coefficients on node degree \cite{lgh-cycles,o-triadic}, rather than on a definition-specific wedge count.}

This presents an opportunity to explore the relationship between connectedness and constraint. As originally defined, constraint decreases with neighborhood size, holding network density constant. A subtle change in definition, from a focus on proportional investment to one on marginal investment, instead produces polynomial {\em growth} in constraint due to a strong interaction effect with local density. In both theoretical \cite{sak-structural} and empirical \cite{rsmob-hierarchical,v-growing} studies, the classical clustering coefficient exhibits the power law relationship
\begin{equation}\label{eq:Ckpowerlaw}
C_\ell\ \propto\ \ell^{-1}\text.
\end{equation}
This may appear to conform to the former definition of constraint, but it actually concerns variation in local density. The family of measures encoded in Def.~\ref{def:Chat} may likewise be expected to behave differently, depending on the variety of TC they measure.
% \cite{dgm-pseudofractal,rsmob-hierarchical}

Taking \(\C\) to be \(\T\) and taking the trivial quotient by \(=\) effectively weights the local connectivity of \(i\), as measured by the wedge count at \(i\), by the number of \(i\)'s neighbors {\em and} the multiplicity of \(i\)'s shared events with them, moderated by the extent of overlap of these events among the neighbors. As a measure of constraint, then, \(C\opsahl_\ell\) is highly sensitive to compounding constraint by multiple events, even between the same small subset of \(i\)'s neighbors. In contrast, \(C\excl_\ell\) is sensitive only to pairs of \(i\)'s neighbors with at least one exclusive common event each (due to the restriction to \(\overline{\T}\)) and is equally sensitive to constraints on \(i\)'s strategic position with respect to any such pair (Thm.~\ref{thm:binning}). That is, \(C\opsahl_\ell\) measures constraint weighted according to the strengths of the relationships (multiplicity of events) between \(i\) and two of their neighbors, while \(C\excl_\ell\) measures constraint in the form of channels of exchange, hidden from \(i\), between neighbors having their own exclusive channels with \(i\).

\begin{figure}
\centerline{
\includegraphics[width=\textwidth]{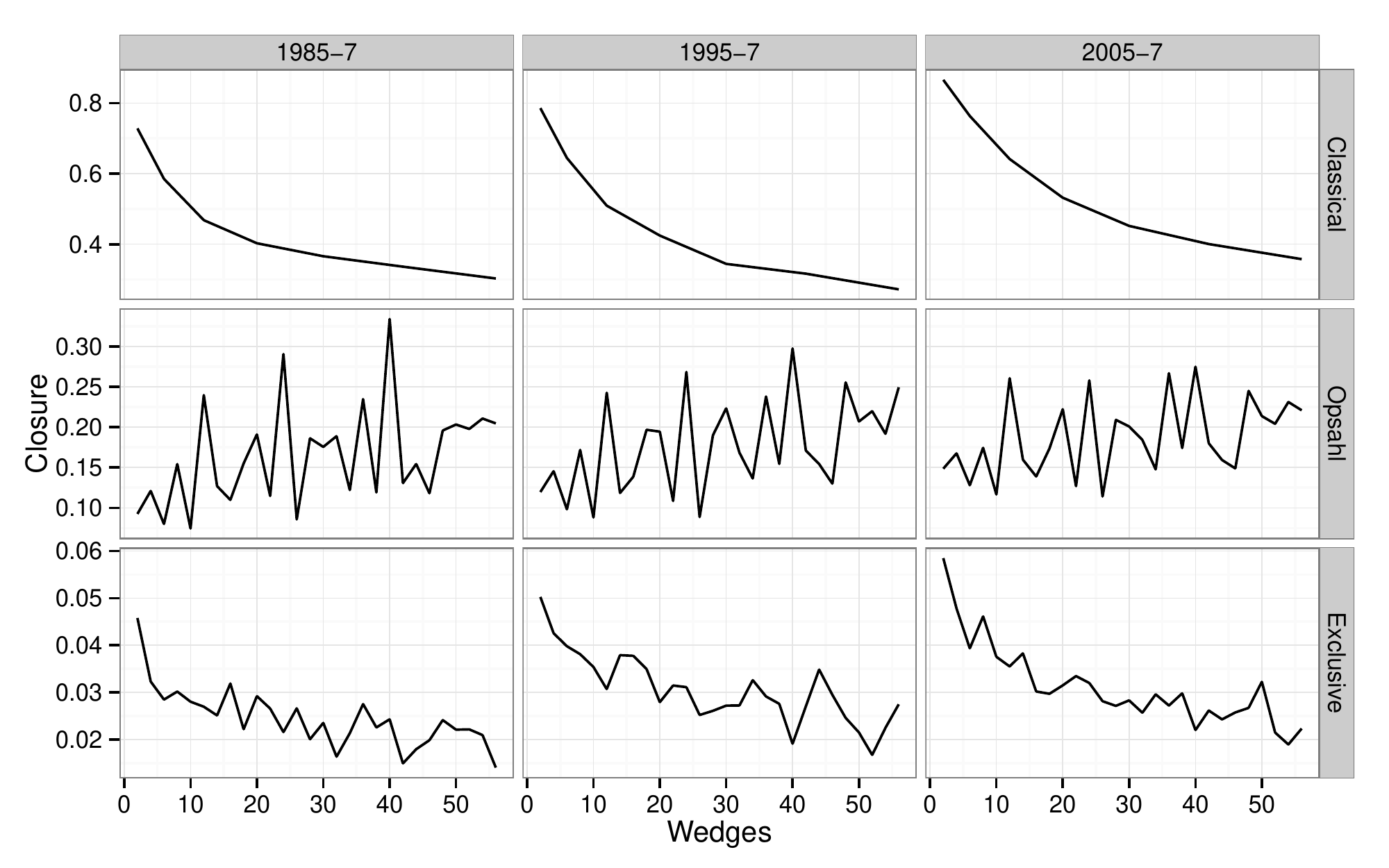}
}
\caption{Three wedge-dependent local clustering coefficients in MR. Note that \(C_\ell\) is only defined when \(\ell=k(k-1)\) for some integer \(k\).}
\label{fig:mr-dep}
\end{figure}

Fig.~\ref{fig:mr-dep} depicts \(C_\ell\), \(C\opsahl_\ell\), and \(C\excl_\ell\) on MR, taken over the same three intervals as in Sec.~\ref{sec:stc}.\footnote{Scatterplots of values in DG1 and GWF are included in the supplement.} \(C_\ell\) follows the expected power law--shaped curve, which persists over time. In contrast, the long-term trend in \(C\opsahl_\ell\) is upward, and exhibits large fluctuations with persistent peaks (e.g.\ at \(\ell=12\) and \(\ell=24\)), an expected artifact of biclique proliferation.\footnote{Whenever \(n\geq 3\) and \(m\geq 2\), the biclique \(K_{n,m}\) yields, for each of its actors \(j\), pairs of neighbors and \(m(m-1)\) ordered pairs of events they share with \(j\), resulting in \((n-1)(n-2)\times m(m-1)\) 4-paths centered at \(j\). When \(m\geq 3\), each of these is closed. Thus, any otherwise untied actor in a copy of \(K_{n,m}\) contributes the atypically high value \(C\opsahl(j)=1\) to the mean \(C\opsahl_\ell\), where \(\ell=(n-1)(n-2)\times m(m-1)\). These values \(\ell=(3-1)(3-2)\times 3(3-1)=12\), \(\ell=(3-1)(3-2)\times 4(4-1)=24\), \(\ell=(4-1)(4-2)\times 3(3-1)=36\), and \(\ell=(3-1)(3-2)\times 5(5-1)=40\) correspond to the highest peaks of \(C\opsahl_\ell\) up to \(\ell=56\). Two clustering coefficients based on \(\widetilde\T/\simeq\) and \(\widetilde\T/\approx\) exhibited similarly expected fluctuations but decreased with wedge count. One based on \(\overline\T/=\) exhibited no such fluctuations and no long-term trend.} \(C\excl_\ell\) mimics \(C_\ell\): The long-term trend is downward and concave, and the fluctuations are modest and transient. Thus, in the world of researhc collaboration, the strengthening of one's (existing) collaborative ties may have a positive effect on the strengths of ties among one's collaborators; while the accumulation of new, mutually-exclusive ties is associated with fewer, on average, ties among them from which oneself is excluded.

Under the assumption that multiple shared events compound and interact to produce many multiple brokerage opportunities, the associated measure of constraint \(C_\ell\opsahl\) may compound enough in kind to outpace it. In these terms, it is not necessarily to \(i\)'s advantage to accumulate neighbors through attendance at common events. In contrast, the constraint \(C_\ell\excl\) imposed by exclusive channels among \(i\)'s neighbors diminishes with increased brokerage opportunities through \(i\)'s own exclusive channels. As in the classical case, therefore, it is unambiguously to \(i\)'s advantage to maintain many neighbors through mutually exclusive channels. These results demonstrate the range of possible behaviors for a custom measure of constraint, and the importance of specifying the brokerage patterns of interest. 

%The results are mostly consistent with expectations: \(C\opsahl\) may be an appropriate (though still coarse) measure of marginality constraint, and \(C^+\) or \(C\excl\) one of proportionality constraint; though wedge-dependent values \(C\opsahl_\ell\) must be compared with caution, due to the influence of bicliques. The structure of MR exhibits two clear trade-offs: First, as observed in other settings, the proportional connectedness (in terms of \(C\)) among a researcher's collaborators diminishes with their number. Second, \(C\excl_\ell\) shows that the level of {\em exclusive} connectedness among a researcher's collaborators diminishes with the number of {\em irredundant} pairings among them. Since exclusivity is important to the theory of structural holes, this may point to a more proper test of that theory for ANs.

\paragraph{Constraint and influence}\label{sec:cent}

Like early conceptions of constraint, the preceding analysis focused on the structure of an actor \(i\)'s neighborhood. Yet much importance has also been placed on actors' positions within the entire network, as popular conceptions of centrality---closeness, betweenness, and eigenvector---attest. This last analysis attempts to discern whether the observed trade-offs are local or global phenomena, via a different extension of the same classical relationship.

Social influence is often measured by eigenvector centrality, based on the recursive principle that an actor accumulates influence through connections with other influential actors \cite{f-centrality2,bh-analyzing}. The eigenvector centrality of \(i\) can be expressed as the cumulative influence of \(i\) through walks (paths that may repeat nodes and edges) of at most some specified length; 1-walk centrality, for instance, equals node degree. This calculation can be inverted to produce a measure of influence through walks of {\em at least} some length \cite{b-simultaneous}: If the \(\ell\)-walk centrality scores of the nodes of an AN \(G\) constitute the vector \({\bf c}_\ell=(c_\ell(1),\ldots,c_\ell(n))\), and the eigenvector centrality scores comprise \({\bf c}_\infty\),\footnote{Here each \({\bf c}={\bf c}_\ell,{\bf c}_\infty\) is normalized so that \(\sum_i{\bf c}(i)^2=1\).} then the {\df \(\ell\)-walk--corrected centrality scores}, which may be positive or negative, constitute \({\bf c}_\infty-{\bf c}_\ell\). The actors' 2-walk centrality scores provide a measure of the local component of their influence that is self-contained, i.e.\ that does not depend on the measure of constraint. Their 2-walk--corrected centrality scores measure the global component.

Each of \(i\)'s neighbors is accessible to \(i\) via some 2-walk, so that \(c_2(i)\) may depend largely on the number of \(i\)'s neighbors. As the previous analysis revealed, however, how these 2-walks are counted is also important. The 2-walks from \(i\) are most closely related to the wedges of \(C\opsahl\), so it is reasonable to expect only a weak relationship between \(c_2(i)\) and \(C\opsahl(i)\). In contrast, \(C\) and \(C\excl\) are insensitive to redundant 2-walks (from \(i\) to some neighbor \(j\)). In order to decompose the relationship between constraint and influence, each clustering coefficient is considered versus each component (2-walk and 2-walk--corrected) of influence.

\begin{figure}
\centerline{\includegraphics[width=\textwidth]{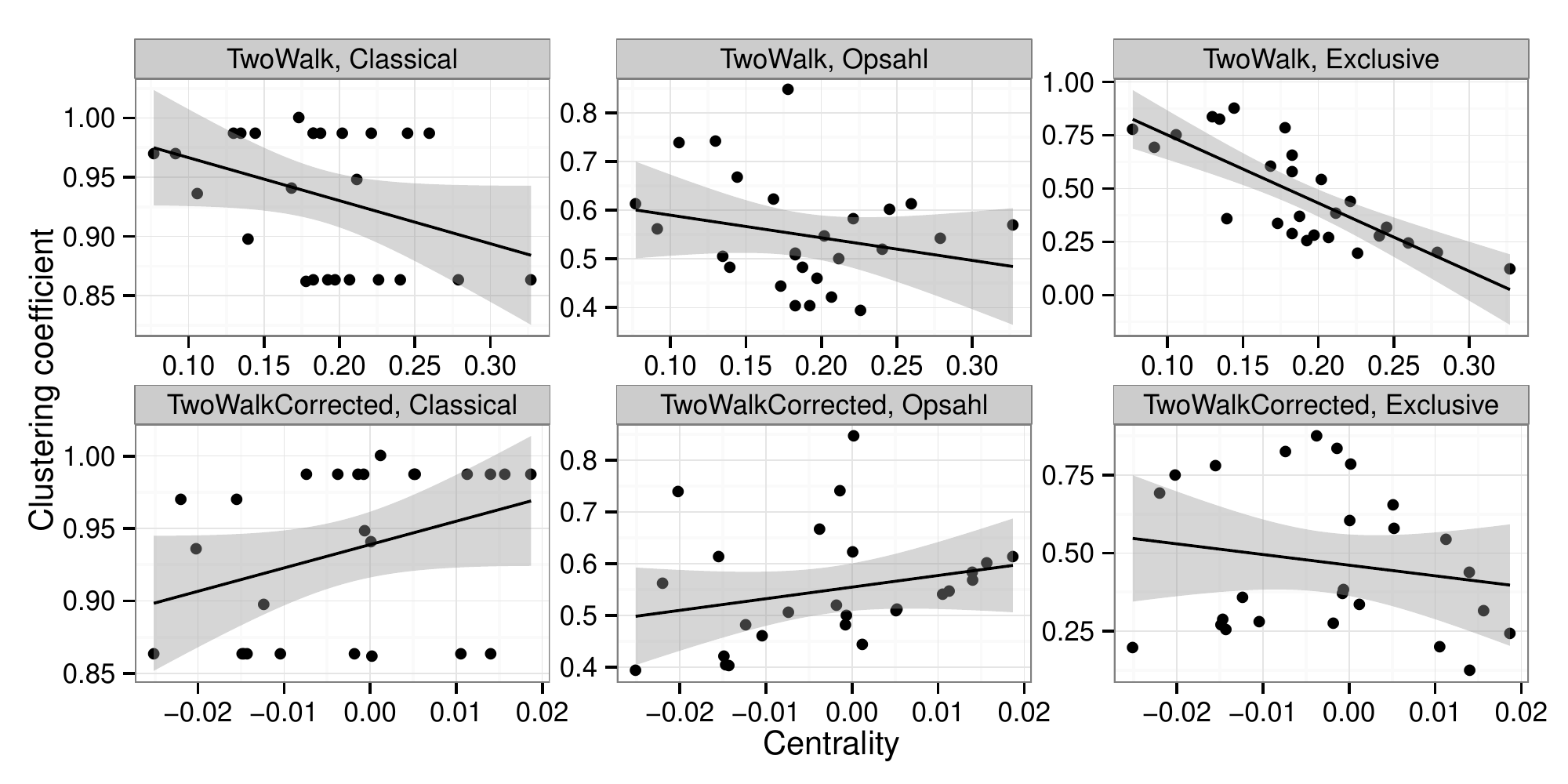}}
\caption{Scatterplots of Opsahl and exclusive clustering coefficients versus 2-walk and 4-walk--corrected eigenvector centrality scores across actors in GWF. Least-squares regression lines and 95\% confidence bands are overlaid.}
\label{fig:cent-gwf}
\end{figure}

Fig.~\ref{fig:cent-gwf} plots the relationships for the CEOs of GWF. (Those for the women of DG1, included in the supplement, are similar.) Those with \(C\opsahl\) are indeed weak, as are those with \(c_\infty-c_2\). The standout is \(C\excl\) versus \(c_2\), and this holds too in DG1: In these small networks, at least, exclusive TC is associated with discernibly lower local influence. Specifically, an increase in 2-walk centrality of \(0.1\) corresponds to a decrease of \(0.46\) (GWF) or \(0.31\) (DG1) in \(C\excl\). The lack of any discernible relationship with 2-walk--corrected centrality suggests that the configuration of \(i\)'s neighborhood is only weakly, if at all, related to \(i\)'s extended influence.

\section{Conclusion}\label{sec:conclusion}

This study pursued a measure of triadic closure for affiliation networks, modeled as bipartite graphs, that controls for the proliferation of bicliques. Bicliques arise from attendance at multiple events by subsets of actors, which is unlikely to reflect the popular understanding of triadic closure. The need for such a measure follows from the sensitivity of existing measures to such structures, even those that control for the sizes of events. In addition to the proposed exclusive clustering coefficient \(C\excl\), the paper presented a classification scheme for affiliation network triads and an axiomatic framework for defining affiliation network clustering coefficients.

An instrumental analysis found \(C\excl\) to measure distinct properties from the classical \(C\) and the recent proposal \(C\opsahl\), and suggested that, in some settings, \(C\excl\) approximates triadic closure as it is characterized over time. An investigation of several empirical affiliation networks revealed patterns of triadic closure much richer than could be inferred from the classical triad census and \(C\) applied to their actor projections. In the author's judgment, \(C\excl\) comes across as a useful counterpoint to \(C\); the two could be viewed as limiting cases between which other clustering coefficients like \(C\opsahl\) interpolate \cite{skokk-generalizations}.

The study has several limitations, most notably the limited number of empirical (and lack of simulated) affiliation networks investigated, and the fact that these networks were constructed using different data collection methods. This leaves the conclusions drawn here open to challenge. Also, no fast algorithms were provided, and the implementations used were not designed for efficiency; applications of the tools described here to large networks will require both.

The tools suggest several other avenues for future work. The classification of affiliation network triads provides the basis for a state transition analysis, which may aide models of network evolution. Affiliation networks also exist with weighted edges, and the generic clustering coefficient described could be adapted, like its predecessor \(C\opsahl\), to this setting.

In summary, it is hoped that the present paper provides a useful framework for the triadic analysis of affiliation networks.

\paragraph{Acknowledgments}

The author thanks Ritchie C. Vaughan for suggesting this line of inquiry; to Miranda Lynch, Tina Eliassi-Rad, Roldan Pozo, Paola Vera-Licona, Linton Freeman, Kathy O'Hara, and Reinhard Laubenbacher for helpful conversations; to Barry Brunson and Pansy Waycaster for several rounds of proofreading; to four anonymous reviewers for highly incisive and supportive comments; and to the Virginia Bioinformatics Institute, the AMS, and UConn Health for data and resources. This project builds upon work done by participants in the Summer 2010 REU in Modeling and Simulation in Systems Biology.

\appendix

\bibliographystyle{nws}
\bibliography{triadic-arxiv}

\newpage

\section{Supplement}

Figs.~\ref{fig:stability}--\ref{fig:discriminability} elaborate upon the scores in Table~3.
Table~\ref{tab:gwflocal} is the counterpart, for GWF, to Table~\ref{tab:dg1local} in the main text.
Fig.~\ref{fig:ex-dep} is the counterpart, for DG1 and GWF, to Fig.~\ref{fig:mr-dep} in the main text, except that the ordered pair for every actor is plotted, rather than their wedge-dependent averages.
Fig.~\ref{fig:cent-dg1} is the counterpart, for DG1, to Fig.~\ref{fig:cent-gwf} in the main text.

\begin{figure}[h]
\centerline{\includegraphics[width=\textwidth]{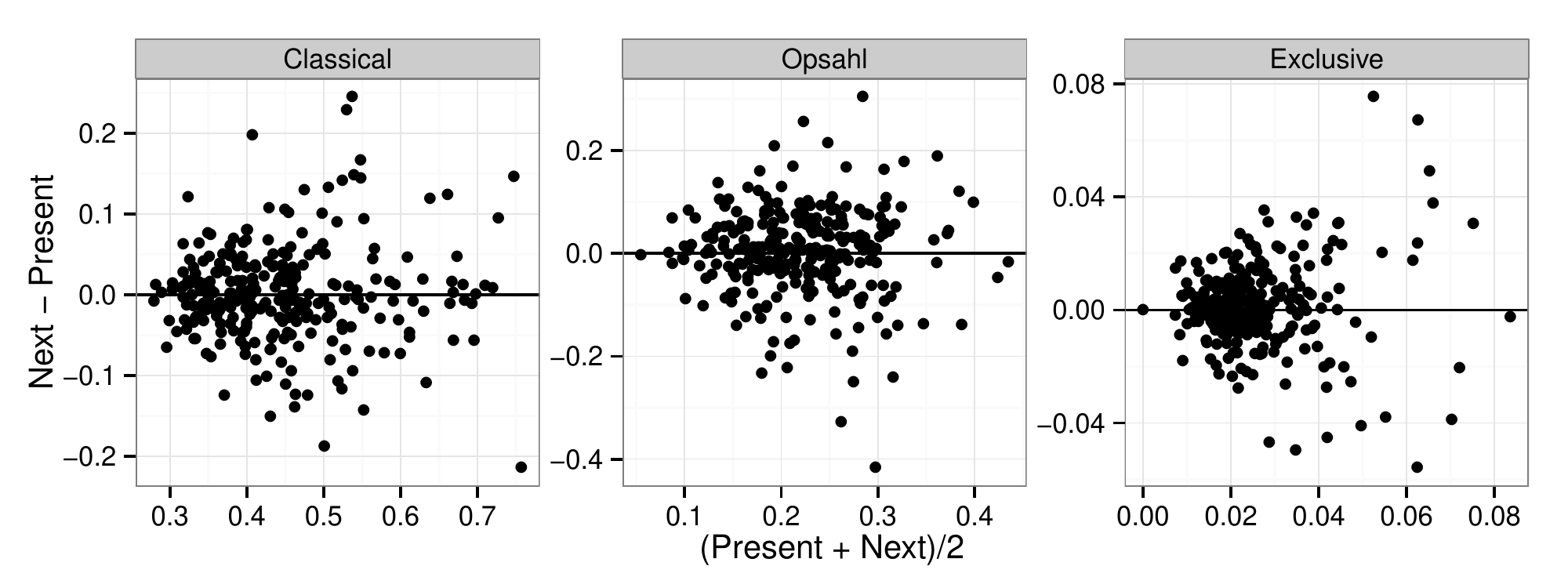}}
\caption{Mean--difference plots for values of \(C\), \(C\opsahl\), and \(C\excl\), taken across 39 subnetworks of MR over 7 pairs of adjacent intervals.}
\label{fig:stability}
\end{figure}

\begin{figure}[h]
\centerline{\includegraphics[width=\textwidth]{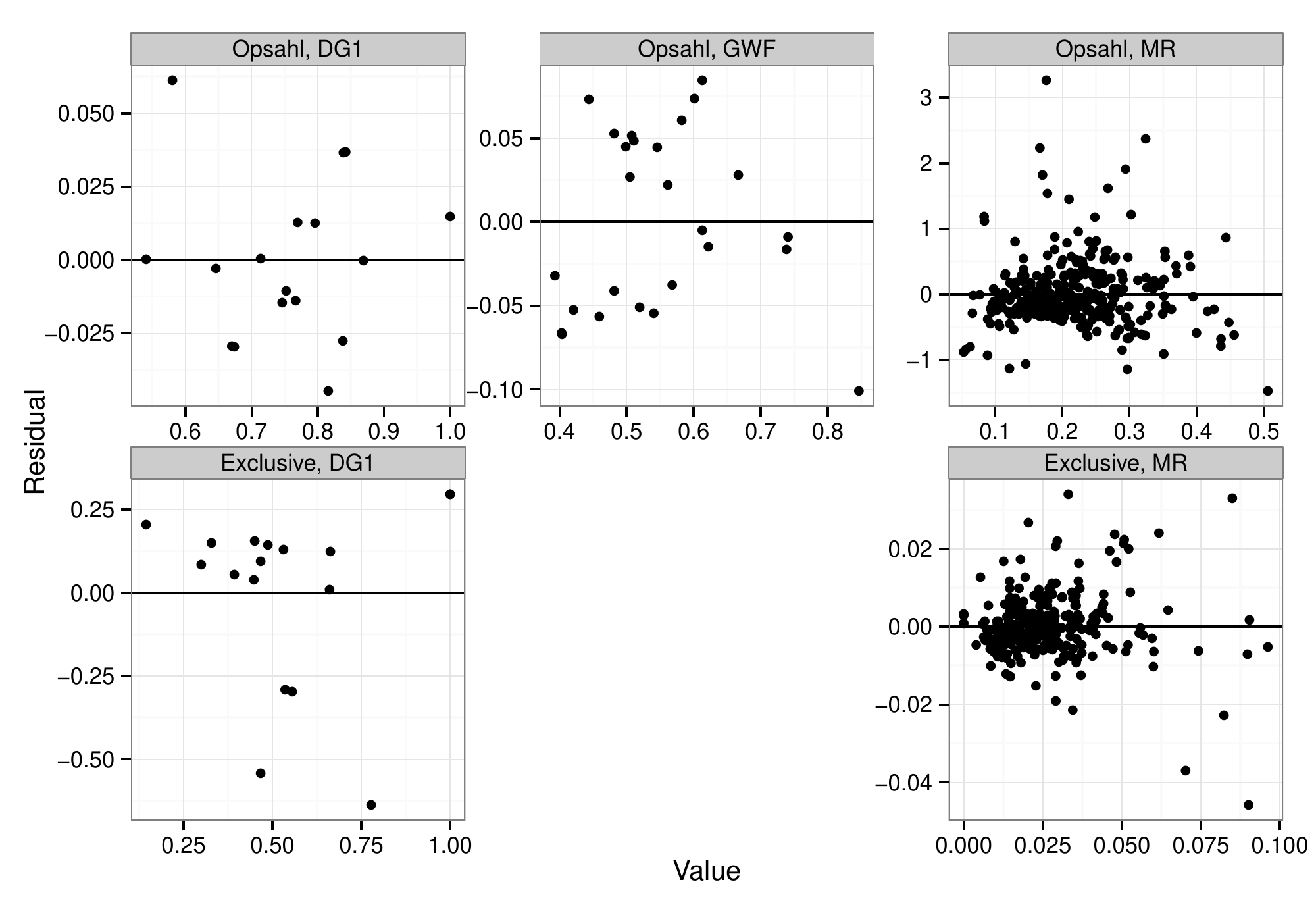}}
\caption{Residual plots for \(C/C_{\rm rand}\) regressed on \(C\opsahl\) and \(D\) regressed on \(C\excl\), taken across the women of DG1, the CEOs of GWF, and 39 subnetworks of MR over 8 intervals.}
\label{fig:validity}
\end{figure}

\begin{figure}[h]
\centerline{\includegraphics[width=\textwidth]{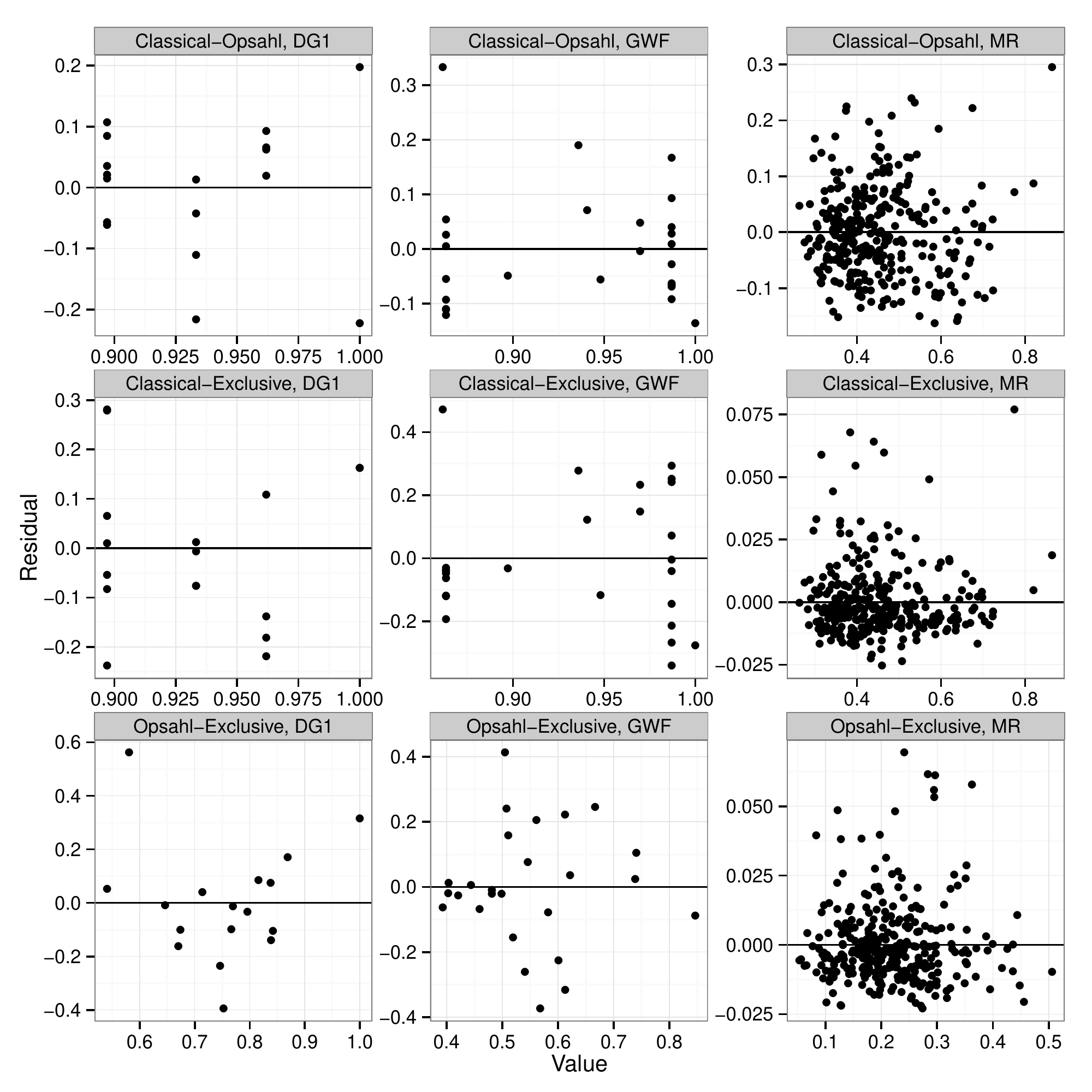}}
\caption{Residual plots for \(C\opsahl\) regressed on \(C\), \(C\excl\) regressed on \(C\), and \(C\excl\) regressed on \(C\opsahl\), taken across the women of DG1, the CEOs of GWF, and 39 subnetworks of MR over 8 intervals.}
\label{fig:distinguishability}
\end{figure}

\begin{figure}[h]
\centerline{\includegraphics[width=\textwidth]{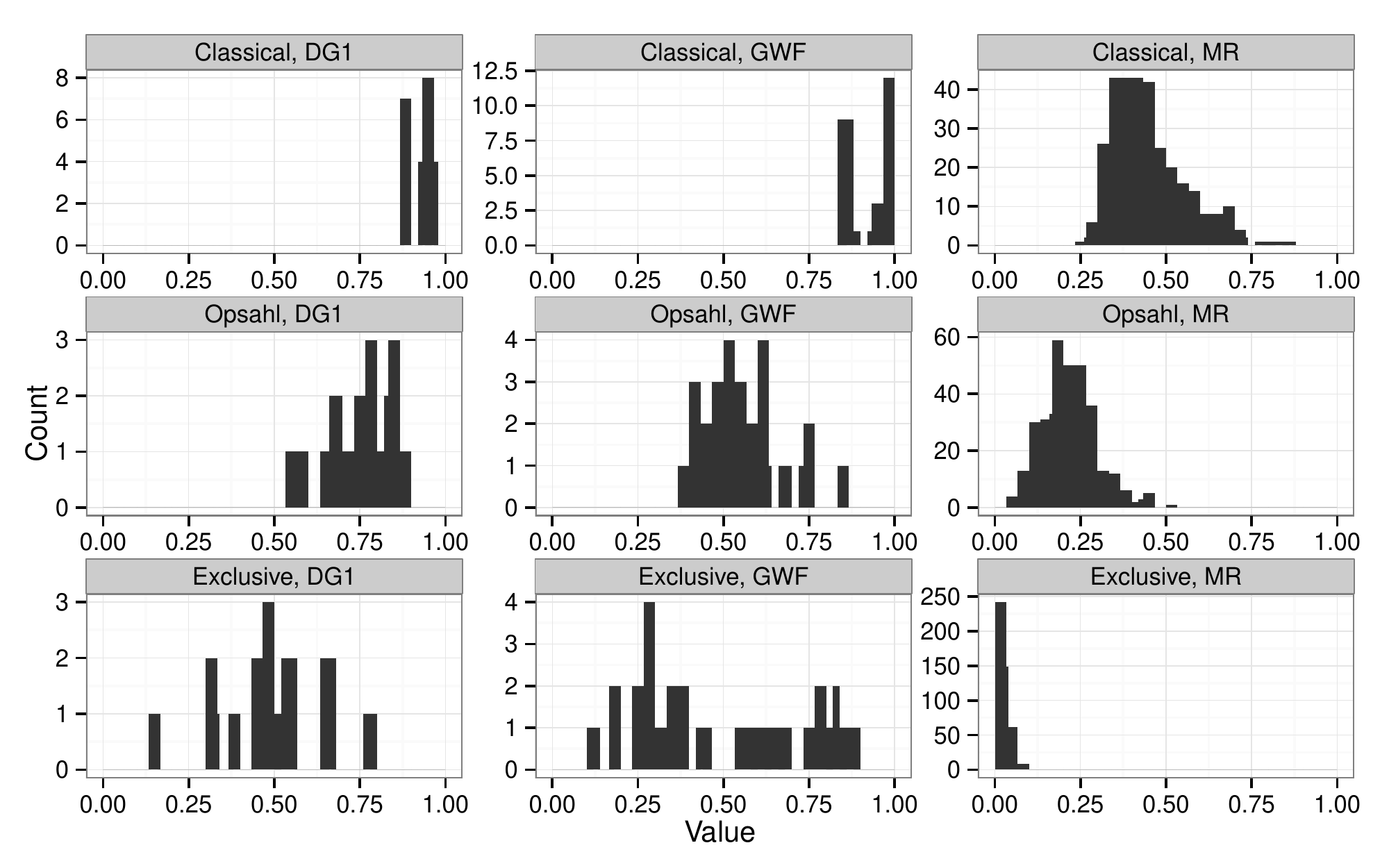}}
\caption{Histograms of values of \(C\), \(C\opsahl\), and \(C\excl\), taken across the women of DG1, the CEOs of GWF, and 39 subnetworks of MR over 8 intervals.}
\label{fig:discriminability}
\end{figure}

\begin{table}
  \caption{Measures of local triadic closure and centrality in GWF.}
  \label{tab:gwflocal}
  \begin{minipage}{\textwidth}
  % latex table generated in R 3.1.0 by xtable 1.7-4 package
% Tue Jun  2 14:48:33 2015
\begin{tabular}{lrrrrrr}
  \hline
\hline
 & Classical & Opsahl & Exclusive & TwoWalk & Eigenvector & TwoWalkCorrected \\ 
  \hline
CEO1 & 0.863 & 0.403 & 0.254 & 0.192 & 0.178 & -0.014 \\ 
  CEO2 & 0.897 & 0.481 & 0.357 & 0.139 & 0.127 & -0.012 \\ 
  CEO3 & 0.987 & 0.741 & 0.833 & 0.130 & 0.128 & -0.001 \\ 
  CEO4 & 0.987 & 0.546 & 0.542 & 0.202 & 0.213 & 0.011 \\ 
  CEO5 & 0.987 & 0.667 & 0.875 & 0.144 & 0.140 & -0.004 \\ 
  CEO6 & 1.000 & 0.444 & 0.333 & 0.173 & 0.174 & 0.001 \\ 
  CEO7 & 0.863 & 0.460 & 0.280 & 0.197 & 0.187 & -0.010 \\ 
  CEO8 & 0.970 & 0.561 & 0.692 & 0.091 & 0.069 & -0.022 \\ 
  CEO9 & 0.936 & 0.739 & 0.750 & 0.106 & 0.086 & -0.020 \\ 
  CEO10 & 0.987 & 0.505 & 0.824 & 0.135 & 0.127 & -0.007 \\ 
  CEO11 & 0.987 & 0.481 & 0.368 & 0.188 & 0.187 & -0.001 \\ 
  CEO12 & 0.970 & 0.613 & 0.778 & 0.077 & 0.061 & -0.016 \\ 
  CEO13 & 0.863 & 0.421 & 0.270 & 0.207 & 0.192 & -0.015 \\ 
  CEO14 & 0.863 & 0.568 & 0.123 & 0.327 & 0.341 & 0.014 \\ 
  CEO15 & 0.987 & 0.601 & 0.315 & 0.245 & 0.261 & 0.016 \\ 
  CEO16 & 0.948 & 0.499 & 0.381 & 0.212 & 0.211 & -0.001 \\ 
  CEO17 & 0.987 & 0.613 & 0.241 & 0.260 & 0.278 & 0.019 \\ 
  CEO18 & 0.861 & 0.847 & 0.784 & 0.178 & 0.178 & 0.000 \\ 
  CEO19 & 0.863 & 0.393 & 0.196 & 0.226 & 0.201 & -0.025 \\ 
  CEO20 & 0.863 & 0.541 & 0.198 & 0.279 & 0.289 & 0.011 \\ 
  CEO21 & 0.863 & 0.404 & 0.286 & 0.183 & 0.168 & -0.015 \\ 
  CEO22 & 0.941 & 0.622 & 0.604 & 0.168 & 0.168 & 0.000 \\ 
  CEO23 & 0.987 & 0.582 & 0.438 & 0.221 & 0.235 & 0.014 \\ 
  CEO24 & 0.863 & 0.519 & 0.275 & 0.240 & 0.239 & -0.002 \\ 
  CEO25 & 0.987 & 0.508 & 0.654 & 0.183 & 0.188 & 0.005 \\ 
  CEO26 & 0.987 & 0.511 & 0.577 & 0.183 & 0.188 & 0.005 \\ 
   \hline
\hline
\end{tabular}

  \end{minipage}
\end{table}

\begin{figure}
\centerline{
\includegraphics[width=\textwidth]{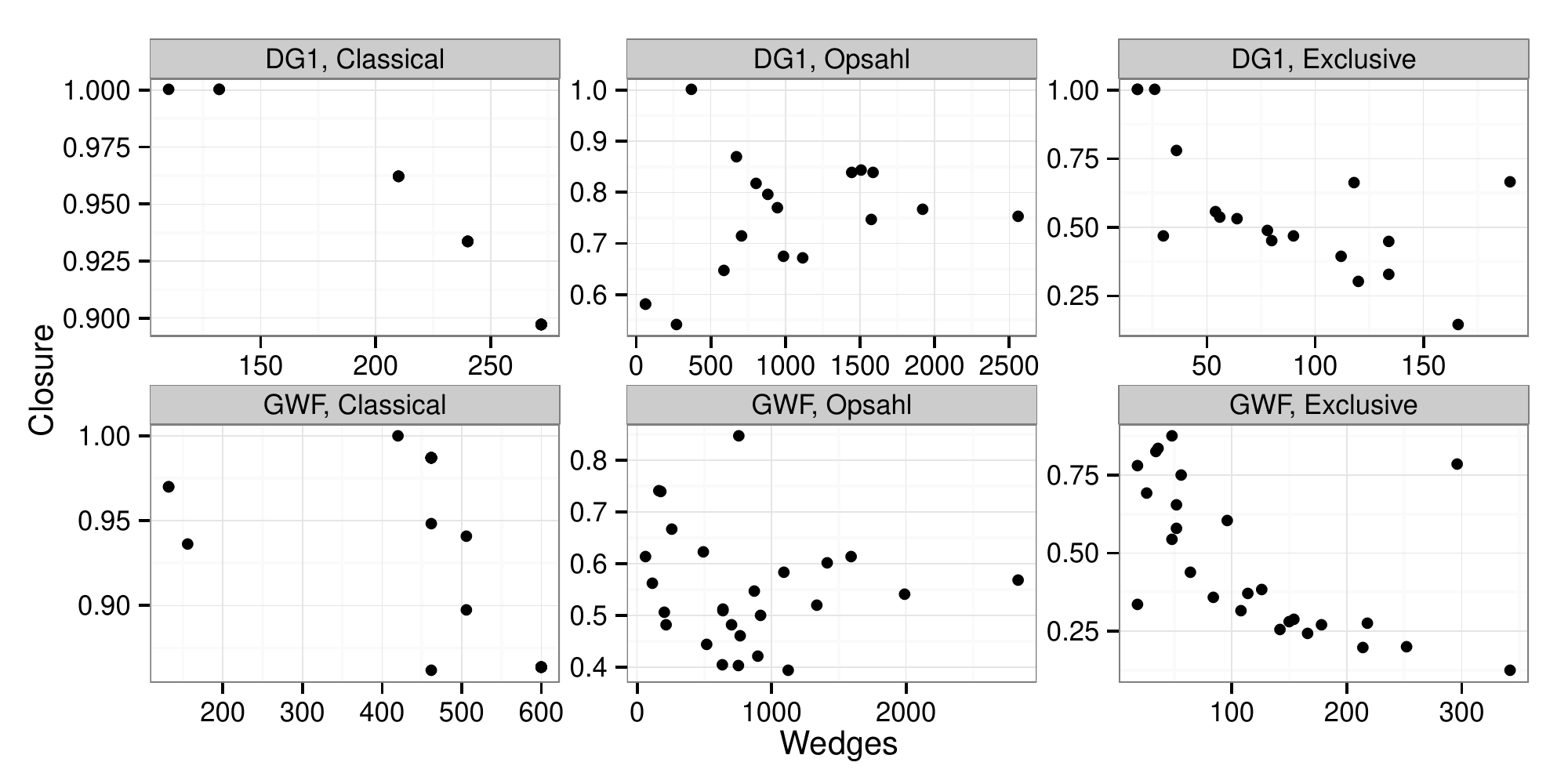}
}
\caption{Four wedge-dependent local clustering coefficients in DG1 and GWF.}
\label{fig:ex-dep}
\end{figure}

\begin{figure}
\centerline{\includegraphics[width=\textwidth]{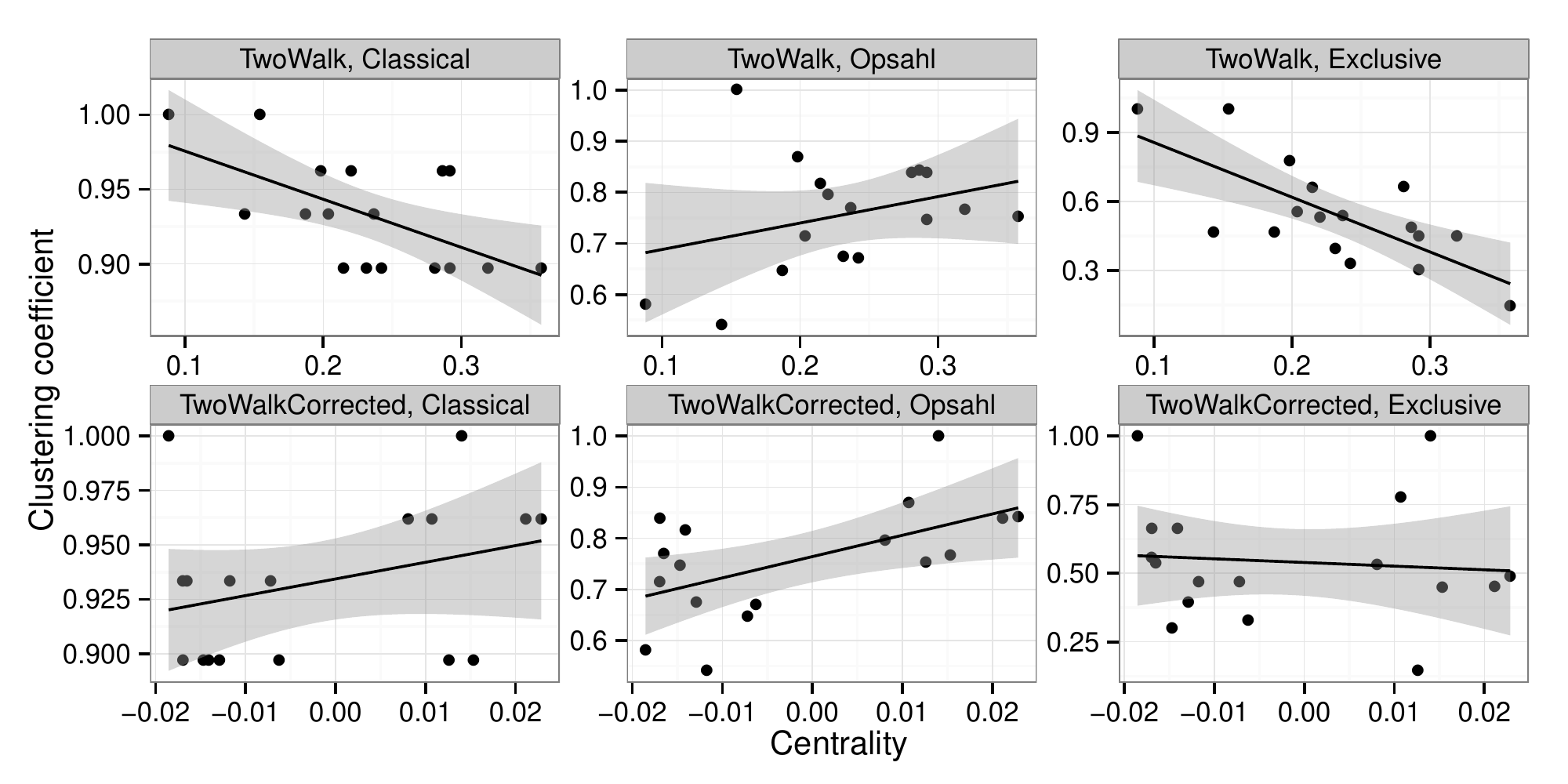}}
\caption{Scatterplots of Opsahl and exclusive clustering coefficients versus 2-walk and 4-walk--corrected eigenvector centrality scores across actors in DG1. Least-squares regression lines and 95\% confidence bands are overlaid.}
\label{fig:cent-dg1}
\end{figure}

\label{lastpage}

\end{document}